\documentclass[a4paper,11pt]{article}

\usepackage{amsmath,amssymb,amsthm,mathtools,bm}
\usepackage[margin=1in]{geometry}
\usepackage{booktabs,array}
\usepackage{graphicx}
\usepackage[ruled,vlined]{algorithm2e}
\usepackage{xcolor}
\usepackage{hyperref}
\hypersetup{colorlinks=true,linkcolor=blue,citecolor=red,urlcolor=cyan}
\usepackage{subcaption}

\numberwithin{equation}{section}

\theoremstyle{plain}
\newtheorem{theorem}{Theorem}[section]
\newtheorem{proposition}[theorem]{Proposition}

\newtheorem{corollary}[theorem]{Corollary}
\theoremstyle{definition}
\newtheorem{definition}[theorem]{Definition}
\newtheorem{assumption}[theorem]{Assumption}
\theoremstyle{remark}
\newtheorem{remark}[theorem]{Remark}

\newcommand{\R}{\mathbb{R}}

\newcommand{\bx}{\mathbf{x}}
\newcommand{\by}{\mathbf{y}}

\newcommand{\bu}{\mathbf{u}}

\newcommand{\bX}{\mathbf{X}}

\newcommand{\bW}{\mathbf{W}}
\newcommand{\be}{\mathbf{e}}
\newcommand{\bn}{\mathbf{n}}
\newcommand{\ba}{\mathbf{a}}

\newcommand{\bA}{\mathbf{A}}
\newcommand{\bB}{\mathbf{B}}
\newcommand{\bC}{\mathbf{C}}
\newcommand{\bD}{\mathbf{D}}
\newcommand{\bF}{\mathbf{F}}

\newcommand{\bI}{\mathbf{I}}

\newcommand{\bQ}{\mathbf{Q}}
\newcommand{\bR}{\mathbf{R}}

\newcommand{\bzero}{\mathbf{0}}

\newcommand{\bDelta}{\boldsymbol{\Delta}}
\newcommand{\bSigma}{\boldsymbol{\Sigma}}

\newcommand{\bPsi}{\boldsymbol{\Psi}}

\newcommand{\etaVec}{\boldsymbol{\eta}}

\newcommand{\E}{\mathbb{E}}
\newcommand{\mean}{\mathbf{m}}
\newcommand{\Cov}{\operatorname{Cov}}
\newcommand{\Var}{\operatorname{Var}}

\newcommand{\lmax}{\lambda_{\max}}

\newcommand{\flow}[2]{\Phi_{#1}^{#2}}
\newcommand{\FTLE}{\operatorname{FTLE}}

\title{A Moment-Based Eulerian Method for Variance-Based Finite-Time Lyapunov Exponent Computation in Stochastic Flows}
\author{Shingyu Leung}
\date{}

\begin{document}
\maketitle

\begin{abstract}
Variance-based finite-time Lyapunov exponents (vFTLEs) provide a stochastic analogue of deterministic FTLE by measuring the covariance of stochastic arrival locations.  Existing PDF-based formulations compute this covariance by solving a Fokker--Planck equation for each initial point, which becomes expensive when the diagnostic is required on a dense grid.  In this work, we develop a moment-based Eulerian approximation to vFTLE in the small-noise regime.  Starting from a stochastic trajectory expansion about the deterministic flow, we derive a closed covariance equation for the leading stochastic displacement.  By embedding this trajectory-wise covariance dynamics into physical space, we obtain an Eulerian transport--reaction equation for a symmetric covariance tensor field.  The covariance associated with each initial point is recovered by evaluating this tensor field at the deterministic arrival location, and a moment-based vFTLE is then defined from its largest eigenvalue.

The proposed method replaces a family of Fokker--Planck solves by the evolution of a single covariance tensor field, requiring only \(d(d+1)/2\) scalar fields in \(d\) dimensions.  It also retains directional information through the eigenvectors of the covariance tensor, allowing the dominant directions of stochastic spreading to be visualized.  We establish the leading-order consistency of the method with PDF-based vFTLE in the small-noise limit, clarify its relation to scalar stochastic sensitivity, and show how the same covariance equation connects process-noise spreading with deterministic deformation.  In particular, deterministic FTLE is recovered, up to an additive constant, from an isotropic initial covariance when no process noise is present, while continuous process noise produces a time-integrated deformation covariance.  Numerical examples include a linear verification problem with an exact covariance solution, an isotropic double-gyre benchmark, and a double-gyre flow with spatially and temporally varying anisotropic diffusion.  The results demonstrate that the proposed covariance moment formulation captures the main uncertainty-induced spreading structures while substantially reducing the computational cost of vFTLE evaluation.
\end{abstract}

\section{Introduction}

Finite-time Lyapunov exponents (FTLEs) are widely used to quantify finite-time stretching and to visualize Lagrangian coherent structures in unsteady flows.  Given a deterministic velocity field \(\bu(\bx,t)\), the associated flow map \(\Phi_{t_0}^{t_0+T}\) sends an initial point \(\bx_0\) to its arrival location at time \(t_0+T\).  The classical FTLE is defined from the largest singular value of the flow-map Jacobian, and its ridge structures have been used extensively as indicators of attracting or repelling material behavior \cite{halyua00,hal01,hal02,shalekmar05,hal15}.  A substantial body of numerical work has also focused on efficient Eulerian and flow-map-based strategies for computing such diagnostics on fixed grids \cite{leu11,leu13,youleu18,youleu20,lywn19}.

In many applications, however, the velocity field is not known exactly.  Uncertainty may arise from measurement noise, unresolved scales, stochastic forcing, model reduction, interpolation error, or sparse and incomplete observations.  These considerations have motivated uncertainty-aware extensions of FTLE-type diagnostics, including variance-based indicators and Fokker--Planck formulations for stochastic flows \cite{sfrs12,youleu21}.  In the PDF-based setting, each initial point is associated with a probability density describing the distribution of possible stochastic arrival locations.  The variance-based FTLE, or vFTLE, then uses the covariance matrix of this arrival density to measure the dominant magnitude and direction of stochastic spreading.  In this sense, vFTLE is a second-moment analogue of deterministic FTLE: rather than measuring the separation of neighboring deterministic trajectories, it measures how uncertainty in the dynamics spreads an initially localized distribution.

The main computational difficulty is that the PDF-based formulation is expensive.  For each initial grid point \(\bx_0\), one must evolve a Fokker--Planck equation initialized by a Dirac mass at \(\bx_0\), and the covariance used in vFTLE is then extracted from the resulting density at the final time.  Even when operator splitting, adaptive activation, sparse subsampling, or parallel computation is used \cite{ngyouleu22}, the cost is dominated by the need to solve many full density evolution problems.  If the spatial grid has \(M\) points and the diagnostic is desired at all initial grid points, the full PDF-based approach requires, in principle, a family of \(M\) density solves.

This observation motivates the central question of this paper: if the vFTLE diagnostic ultimately depends only on the covariance of the stochastic arrival distribution, is it necessary to compute the full probability density?  We show that, in the small-noise regime, the answer is negative.  Starting from the stochastic dynamics \(d\bX_t=\bu(\bX_t,t)\,dt+\varepsilon\bB(\bX_t,t)\,d\bW_t\), we use the expansion \(\bX_t=\Phi_{t_0}^{t}(\bx_0)+\varepsilon\etaVec_t+O(\varepsilon^2)\), where \(\etaVec_t\) is the leading stochastic displacement about the deterministic trajectory.  The covariance of \(\etaVec_t\) satisfies a closed matrix equation along deterministic characteristics.  By rewriting this characteristic covariance equation in physical space, we obtain the Eulerian covariance transport equation
\begin{equation}
    \partial_t \bQ+\bu\cdot\nabla\bQ
    =
    (\nabla\bu)\bQ+\bQ(\nabla\bu)^T+\bD,
    \qquad
    \bQ(\bx,t_0)=\bzero,
    \label{eq:eulerian_Q_intro}
\end{equation}
where \(\bD=\bB\bB^T\) is the diffusion tensor.  Once \(\bQ\) is computed, the covariance associated with the initial point \(\bx_0\) is obtained by evaluating this Eulerian tensor field at the deterministic arrival point, namely \(\bC_T(\bx_0)=\bQ(\Phi_{t_0}^{t_0+T}(\bx_0),t_0+T)\).  The proposed moment-based vFTLE is then defined by \(\mu_{t_0}^{t_0+T}(\bx_0)=|T|^{-1}\log\sqrt{\lambda_{\max}(\varepsilon^2\bC_T(\bx_0))}\).

The resulting method replaces the evolution of many probability density functions by the evolution of a single symmetric covariance tensor field.  In \(d\) spatial dimensions, this requires only \(d(d+1)/2\) scalar fields.  Thus, in two dimensions, the method evolves three scalar covariance components, and in three dimensions it evolves six.  The reduction is substantial when the diagnostic is required over a dense set of initial conditions.  Moreover, because the full covariance tensor is retained, the method provides not only a scalar moment-vFTLE field but also the dominant directions of stochastic spreading through the eigenvectors of \(\bC_T\).

The proposed method should be understood as a leading-order moment approximation to PDF-based vFTLE, not as a replacement for the full Fokker--Planck description in all regimes.  It is most appropriate when the noise amplitude is small and the second moment of the arrival distribution is the primary quantity of interest.  If the full density is needed, if the arrival distribution is strongly non-Gaussian or multimodal, or if higher-order uncertainty information is important, then the Fokker--Planck formulation remains more informative.  The purpose of the present work is instead to provide a computationally efficient covariance-based approximation in settings where vFTLE is used primarily as a second-moment diagnostic.  The stochastic expansion and moment interpretation used here are consistent with standard small-noise and numerical SDE viewpoints \cite{Kloeden92,Higham01}.

The main contributions of this paper are as follows.  First, we derive a leading-order covariance approximation to vFTLE from the small-noise expansion of stochastic trajectories.  Second, we formulate the covariance dynamics as an Eulerian tensor transport equation, reducing the computation to \(d(d+1)/2\) scalar fields.  Third, we combine this covariance field with deterministic flow-map evaluation to define a moment-based vFTLE and its associated dominant spreading directions.  Fourth, we clarify the relation between the proposed diagnostic, full PDF-based vFTLE, stochastic sensitivity, and deterministic deformation.  In particular, we show that deterministic FTLE is recovered from the covariance formulation when uncertainty is placed in the initial condition, and that process noise leads to a time-integrated deformation covariance.  Finally, we present numerical examples, including a linear verification problem with an exact covariance solution, an isotropic double-gyre benchmark, and a double-gyre example with spatially varying anisotropic diffusion.  These examples demonstrate both the accuracy of the covariance moment equation and the additional directional information carried by the tensor formulation.

The remainder of the paper is organized as follows.  Section~\ref{sec:background} reviews deterministic FTLE, PDF-based stochastic vFTLE, and the covariance-moment motivation.  Section~\ref{sec:proposed} presents the proposed moment-based Eulerian approach, including the small-noise expansion, the Eulerian covariance transport equation, the moment-vFTLE definition, and its relation to deterministic deformation, PDF-based vFTLE, and stochastic sensitivity.  Section~\ref{sec:numerical_method} discusses the numerical discretization and computational complexity.  Section~\ref{sec:numerical_experiments} presents the numerical examples.  Conclusions and future directions are given in Section~\ref{sec:conclusion}.

\section{Background}
\label{sec:background}

This section reviews the deterministic and stochastic ingredients needed for the proposed moment formulation.  We first recall the deterministic flow map and the classical FTLE, which provide the baseline finite-time stretching diagnostic.  We then summarize the PDF-based stochastic formulation of vFTLE, where uncertainty in the dynamics is represented by a Fokker--Planck equation and the covariance of the stochastic arrival distribution is used to define a variance-based diagnostic.  Finally, we identify the key observation motivating the present work: although the full probability density contains rich information, vFTLE depends only on its second central moment.  This suggests that, in the small-noise regime, the covariance can be evolved directly without computing the full density.

\subsection{Deterministic flow maps and finite-time Lyapunov exponents}

Let \(\Omega\subset\R^d\) be a computational domain and let \(\bu:\Omega\times[t_0,t_1]\to\R^d\), with \(t_1=t_0+T\), be a sufficiently smooth velocity field.  The deterministic trajectory starting from \(\bx_0\in\Omega\) at time \(t_0\) is governed by
\begin{equation}
    \frac{d\bx}{dt}=\bu(\bx(t),t),
    \qquad \bx(t_0)=\bx_0 .
    \label{eq:det_ode}
\end{equation}
The corresponding flow map is denoted by \(\Phi_{t_0}^{t}(\bx_0)=\bx(t;t_0,\bx_0)\).  It maps the take-off location \(\bx_0\) at time \(t_0\) to the arrival location at time \(t\).

The sensitivity of the arrival location with respect to the initial location is measured by the Jacobian matrix \(\nabla_{\bx_0}\Phi_{t_0}^{t_1}(\bx_0)\), where the gradient is taken with respect to the initial coordinate.  The associated Cauchy--Green deformation tensor is \(\bDelta_{t_0}^{t_1}(\bx_0)=(\nabla_{\bx_0}\Phi_{t_0}^{t_1}(\bx_0))^T\nabla_{\bx_0}\Phi_{t_0}^{t_1}(\bx_0)\).  The deterministic finite-time Lyapunov exponent (FTLE) is defined by
\[
    \sigma_{t_0}^{t_1}(\bx_0)=\frac{1}{|T|}\log\sqrt{\lambda_{\max}\left(\bDelta_{t_0}^{t_1}(\bx_0)\right)}.
\]
This quantity measures the maximal finite-time stretching rate of an infinitesimal perturbation initially placed at \(\bx_0\).  In many flow-visualization applications, pronounced ridges of the FTLE field are used as numerical indicators of coherent structures and transport barriers.

A direct Lagrangian computation of the FTLE usually requires releasing particles from all grid points, integrating the characteristic equations, and then differentiating the numerically obtained flow map.  This approach is conceptually simple, but it may become expensive or inaccurate when many flow maps are needed, when the integration time is long, or when the deformation gradient is highly sensitive to small perturbations.  Eulerian methods offer an alternative by computing flow maps on a fixed mesh.  One common formulation introduces a vector-valued labeling function \(\bPsi(\bx,t)\) satisfying the Liouville equation \(\bPsi_t+\bu\cdot\nabla_{\bx}\bPsi=0\), with the initial condition \(\bPsi(\bx,t_0)=\bx\).  The solution gives the backward flow map, since \(\bPsi(\by,t_1)\) is the take-off location at time \(t_0\) of the particle arriving at \(\by\) at time \(t_1\).  Forward flow maps can also be constructed by composing short-time maps using interpolation.  This Eulerian flow-map machinery is useful in the present work because the proposed moment method also requires evaluating quantities transported along deterministic characteristics.

\subsection{Stochastic dynamics and PDF-based vFTLE}

In practical applications, the velocity field is often affected by measurement noise, unresolved scales, model error, or other sources of uncertainty.  A common model for such uncertainty is the stochastic differential equation
\begin{equation}
    d\bX_t=\bu(\bX_t,t)\,dt+\varepsilon\bB(\bX_t,t)\,d\bW_t,
    \qquad \bX_{t_0}=\bx_0,
    \label{eq:sde}
\end{equation}
where \(\bX_t\in\R^d\) is the stochastic particle position, \(\bW_t\in\R^m\) is a standard Brownian motion, \(\bB(\bx,t)\in\R^{d\times m}\) describes how the noise enters the system, and \(0<\varepsilon\ll 1\) controls the noise strength.  We denote the diffusion tensor by \(\bD(\bx,t)=\bB(\bx,t)\bB(\bx,t)^T\in\R^{d\times d}\).  For a fixed initial location \(\bx_0\), the arrival location \(\bX_t\) is now a random variable rather than a single deterministic point.

Let \(p(\by,t;\bx_0,t_0)\) be the probability density function of \(\bX_t\) conditional on \(\bX_{t_0}=\bx_0\), evaluated at the arrival coordinate \(\by\in\Omega\).  Under the above stochastic model, \(p\) satisfies the Fokker--Planck equation
\begin{equation}
    \partial_t p+\nabla_{\by}\cdot(\bu(\by,t)p)
    =
    \frac{\varepsilon^2}{2}\sum_{i,j=1}^d
    \partial_{y_i y_j}^2\!\left(D_{ij}(\by,t)p\right),
    \qquad
    p(\by,t_0;\bx_0,t_0)=\delta(\by-\bx_0).
    \label{eq:fp}
\end{equation}
This equation describes the evolution of the full distribution of possible particle locations.  If the uncertainty vanishes, i.e. \(\varepsilon=0\), the density reduces formally to a Dirac mass transported by the deterministic flow.

The PDF contains several levels of information.  Its first moment gives the expected arrival location, denoted by \(\mean_{t_0}^{t_1}(\bx_0)=\int_\Omega \by\,p(\by,t_1;\bx_0,t_0)\,d\by\).  The expected-flow-map approach computes an FTLE-type diagnostic by replacing the deterministic flow map \(\Phi_{t_0}^{t_1}\) with this mean arrival map \(\mean_{t_0}^{t_1}\).  This gives a natural extension of the deterministic FTLE, but it only uses the first-order moment of the arrival distribution.  It therefore does not directly measure the spread of the stochastic particles around their mean.

To capture the spreading of the arrival distribution, one may instead use the covariance matrix
\begin{equation}
    \bSigma_{t_0}^{t_1}(\bx_0)
    =
    \int_\Omega
    \left(\by-\mean_{t_0}^{t_1}(\bx_0)\right)
    \left(\by-\mean_{t_0}^{t_1}(\bx_0)\right)^T
    p(\by,t_1;\bx_0,t_0)\,d\by .
    \label{eq:pdf_cov}
\end{equation}
The largest eigenvalue of \(\bSigma_{t_0}^{t_1}(\bx_0)\) gives the strongest principal variance direction of the stochastic arrival locations.  This motivates the variance-based FTLE, or vFTLE,
\begin{equation}
    v_{t_0}^{t_1}(\bx_0)=\frac{1}{|T|}\log\sqrt{\lambda_{\max}\left(\bSigma_{t_0}^{t_1}(\bx_0)\right)}.
    \label{eq:vftle_pdf}
\end{equation}
Compared with the expected FTLE, vFTLE uses second-order information from the PDF and is therefore more directly related to uncertainty-induced spreading.

The PDF-based formulation is mathematically natural, but it is computationally expensive.  To evaluate \(\bSigma_{t_0}^{t_1}(\bx_0)\) for all initial grid points \(\bx_0\), one needs to solve one Fokker--Planck equation for each initial location.  In two spatial dimensions, each solve is an advection--diffusion equation on the physical domain; in three dimensions, the cost is even higher.  Although operator splitting, adaptive refinement, and parallelization can reduce the cost, computing a full PDF when only its covariance is ultimately needed may still be inefficient.

\subsection{Motivation for a covariance moment equation}

The key observation behind the present work is that vFTLE depends on the Fokker--Planck solution only through the covariance matrix of the stochastic arrival location.  Therefore, if our goal is to approximate vFTLE rather than reconstruct the full arrival PDF, it is natural to ask whether the second moment can be evolved directly.  This is the motivation for introducing a moment-based Eulerian method.

The approach developed below is based on a small-noise expansion of the stochastic trajectory.  In the regime \(0<\varepsilon\ll 1\), the stochastic solution can be regarded as a perturbation of the deterministic trajectory.  Formally, one writes \(\bX_t=\Phi_{t_0}^{t}(\bx_0)+\varepsilon\etaVec_t+O(\varepsilon^2)\), where \(\etaVec_t\in\R^d\) describes the leading stochastic displacement from the deterministic path.  Substituting this expansion into the stochastic differential equation and retaining the leading-order stochastic terms gives a linear stochastic equation for \(\etaVec_t\).  Its covariance satisfies a closed deterministic matrix equation involving only the velocity gradient \(\nabla_{\bx}\bu\) and the diffusion tensor \(\bD=\bB\bB^T\).

This observation provides a significant simplification.  Instead of solving a full Fokker--Planck equation for every initial point, we solve a transport equation for a covariance tensor field on the Eulerian grid.  In two dimensions, the covariance tensor is symmetric and hence has only three independent components.  Thus the proposed method reduces the computation to three coupled scalar transport--reaction equations, followed by an evaluation along the deterministic flow map.  The resulting diagnostic is a moment-based approximation to vFTLE.  Its purpose is not to replace the full PDF description in all noise regimes, but to provide an efficient and analytically interpretable approximation when the leading covariance of the stochastic displacement is the dominant quantity of interest.

This perspective also clarifies the relationship between the proposed method and earlier Eulerian approaches.  Deterministic Eulerian FTLE methods compute finite-time deformation through flow maps.  PDF-based methods for uncertain flows compute the full distribution of stochastic arrivals and then extract first- or second-order information.  The present method sits between these two approaches: it keeps the Eulerian flow-map viewpoint, but replaces the full Fokker--Planck density by a closed covariance transport equation.  This leads to a substantially cheaper computation while retaining the second-moment information needed for variance-based flow visualization.

\section{The Proposed Moment-Based Eulerian Approach}
\label{sec:proposed}

This section develops the proposed covariance-based approximation to vFTLE.  We first derive the leading-order covariance dynamics from a small-noise expansion of stochastic trajectories.  We then rewrite the resulting characteristic covariance equation as an Eulerian transport--reaction equation for a symmetric tensor field.  Finally, we define the moment-based vFTLE and place it in context by relating it to deterministic deformation, full PDF-based vFTLE, and scalar stochastic sensitivity.

\subsection{Small-noise expansion and covariance dynamics}

For clarity of exposition, we set \(t_0=0\) in this section and write \(t_1=T\).  The extension to a general initial time \(t_0\) only changes the notation.  Throughout this section, \(\bx(t)=\flow{0}{t}(\bx_0)\) denotes the deterministic trajectory starting from \(\bx_0\), and \(\bA(t)=\nabla_{\bx}\bu(\bx(t),t)\) denotes the velocity gradient evaluated along this trajectory.

\subsubsection{Perturbation expansion}

We consider the stochastic trajectory governed by \eqref{eq:sde} and seek a small-noise expansion around the deterministic path.  Formally, we write
\begin{equation}
    \bX_t=\bx(t)+\varepsilon\etaVec_t+\varepsilon^2\boldsymbol{\zeta}_t+O(\varepsilon^3),
    \label{eq:small_noise_expansion}
\end{equation}
where \(\etaVec_t\) is the leading stochastic displacement and \(\boldsymbol{\zeta}_t\) contains the second-order correction.  Substituting this ansatz into the stochastic differential equation and expanding the drift term gives \(\bu(\bX_t,t)=\bu(\bx(t),t)+\varepsilon\bA(t)\etaVec_t+O(\varepsilon^2)\).  Since the noise term is already multiplied by \(\varepsilon\), its leading contribution is \(\varepsilon\bB(\bx(t),t)\,d\bW_t\).  Comparing terms of order \(\varepsilon\) gives the linear stochastic differential equation
\begin{equation}
    d\etaVec_t=\bA(t)\etaVec_t\,dt+\bB(\bx(t),t)\,d\bW_t,
    \qquad \etaVec_0=\bzero .
    \label{eq:linearized_eta}
\end{equation}
Thus, to leading order, the stochastic deviation from the deterministic trajectory is governed by the variational dynamics of the deterministic flow, forced by the diffusion tensor along the same trajectory.  In particular, \(\etaVec_t\) has zero mean when the initial perturbation is deterministic and zero.  Consequently, the mean of \(\bX_t\) differs from \(\bx(t)\) only at order \(\varepsilon^2\), while its covariance is of order \(\varepsilon^2\).

This expansion is the main reason a covariance-based approximation is possible.  The full Fokker--Planck equation describes the entire probability density of \(\bX_t\), but the leading contribution to the covariance of \(\bX_t\) is already determined by the linearized process \(\etaVec_t\).  Therefore, for small noise, it is sufficient to compute the covariance of \(\etaVec_t\) rather than the full density of \(\bX_t\).

\subsubsection{Covariance equation along characteristics}

Define the leading-order covariance matrix by \(\bC(t;\bx_0)=\E[\etaVec_t\etaVec_t^T]\).  Applying It\^o's product rule to \(\etaVec_t\etaVec_t^T\), using \(\bD=\bB\bB^T\), gives
\begin{equation}
    \frac{d\bC}{dt}
    =
    \bA(t)\bC+\bC\bA(t)^T+\bD(\bx(t),t),
    \qquad
    \bC(0;\bx_0)=\bzero .
    \label{eq:cov_ode}
\end{equation}
This is a nonautonomous matrix Lyapunov equation along the deterministic characteristic.  It is closed: the right-hand side involves only the deterministic trajectory, the velocity gradient, the diffusion tensor, and the covariance matrix itself.  No higher moments of the stochastic process are needed at this order.

The covariance of the original stochastic trajectory satisfies, formally, \(\Cov(\bX_t\mid \bX_0=\bx_0)=\varepsilon^2\bC(t;\bx_0)+O(\varepsilon^3)\), under the smoothness assumptions required for the small-noise expansion.  Hence \(\bC(t;\bx_0)\) captures the leading shape and orientation of the stochastic arrival cloud.  Its largest eigenvalue gives the dominant direction of stochastic spreading, while its eigenvectors describe the corresponding principal axes.

The solution of \eqref{eq:cov_ode} also has a useful representation.  Let \(\bF(t,s;\bx_0)\) be the fundamental matrix of the variational equation along the deterministic trajectory, satisfying \(\partial_t\bF(t,s;\bx_0)=\bA(t)\bF(t,s;\bx_0)\) and \(\bF(s,s;\bx_0)=\bI\).  Then
\begin{equation}
    \bC(t;\bx_0)
    =
    \int_0^t
    \bF(t,s;\bx_0)\,
    \bD(\bx(s),s)\,
    \bF(t,s;\bx_0)^T\,ds .
    \label{eq:cov_representation}
\end{equation}
This formula shows explicitly how uncertainty injected at an earlier time \(s\) is transported and stretched by the deterministic deformation from time \(s\) to time \(t\).  It also implies that \(\bC(t;\bx_0)\) remains symmetric positive semidefinite whenever \(\bD\) is symmetric positive semidefinite.  In the present work we exploit the symmetry of \(\bC\) by evolving only its independent scalar components.

\begin{remark}
If the initial particle location is itself random with covariance \(\bC_0(\bx_0)\), then the initial condition in \eqref{eq:cov_ode} should be replaced by \(\bC(0;\bx_0)=\bC_0(\bx_0)\).  In that case, the representation formula becomes
\[
    \bC(t;\bx_0)
    =
    \bF(t,0;\bx_0)\bC_0(\bx_0)\bF(t,0;\bx_0)^T
    +
    \int_0^t
    \bF(t,s;\bx_0)\bD(\bx(s),s)\bF(t,s;\bx_0)^T\,ds .
\]
Thus the same formulation can incorporate both uncertainty in the initial condition and uncertainty accumulated along the trajectory.
\end{remark}

\begin{remark}
For a linear velocity field \(\bu(\bx,t)=\bA(t)\bx+\ba(t)\) and spatially independent diffusion tensor \(\bD(t)\), equation \eqref{eq:cov_ode} gives the exact covariance of the stochastic process.  For nonlinear flows and state-dependent diffusion, it gives the leading-order covariance in the small-noise limit.  This distinction is important: the method proposed here should be interpreted as a moment approximation to vFTLE in the small-noise regime, rather than as a full replacement for the Fokker--Planck description in all stochastic regimes.
\end{remark}

\subsection{Eulerian covariance transport}

The covariance equation \eqref{eq:cov_ode} was derived along a single deterministic trajectory.  If it were solved separately for every initial point \(\bx_0\), the resulting method would still be essentially Lagrangian.  The goal of this section is to rewrite the same covariance dynamics as an Eulerian transport equation on the physical grid.  This reformulation is the main computational step that makes the proposed moment-based method substantially cheaper than the PDF-based vFTLE approach.

\subsubsection{Definition of the Eulerian covariance field}

For the moment, we continue to set \(t_0=0\).  Let \(\bx(t)=\flow{0}{t}(\bx_0)\) be the deterministic trajectory starting from \(\bx_0\).  The covariance matrix \(\bC(t;\bx_0)\) obtained in \eqref{eq:cov_ode} is naturally indexed by the initial point \(\bx_0\).  To obtain an Eulerian description, we instead store this covariance at the current physical location of the same deterministic particle.  More precisely, we define a matrix-valued field \(\bQ(\bx,t)\) by
\begin{equation}
    \bQ(\flow{0}{t}(\bx_0),t)=\bC(t;\bx_0).
    \label{eq:Q_definition}
\end{equation}
Thus \(\bQ(\bx,t)\) represents the leading-order covariance of stochastic particles whose deterministic reference trajectory arrives at \(\bx\) at time \(t\).  In this sense, \(\bQ\) is an arrival-coordinate covariance field.

Taking the material derivative of \(\bQ\) along the deterministic velocity field gives \(D\bQ/Dt=\partial_t\bQ+\bu\cdot\nabla_{\bx}\bQ\).  Since \(\bQ(\flow{0}{t}(\bx_0),t)=\bC(t;\bx_0)\), this material derivative must agree with \(d\bC/dt\) along the same characteristic.  Substituting \eqref{eq:cov_ode} then gives the Eulerian covariance transport equation
\begin{equation}
    \partial_t\bQ+\bu\cdot\nabla_{\bx}\bQ
    =
    (\nabla_{\bx}\bu)\bQ+\bQ(\nabla_{\bx}\bu)^T+\bD,
    \qquad
    \bQ(\bx,0)=\bzero .
    \label{eq:Q_PDE}
\end{equation}
This tensor-valued transport--reaction equation is the central equation of the proposed method.  The transport term moves covariance information with the deterministic flow, while the reaction terms \((\nabla_{\bx}\bu)\bQ+\bQ(\nabla_{\bx}\bu)^T\) describe the stretching and rotation of the covariance ellipsoid by the local velocity gradient.  The source term \(\bD=\bB\bB^T\) injects new stochastic variance into the system.

The important computational point is that \eqref{eq:Q_PDE} is solved only once on the Eulerian grid.  This is in sharp contrast with the Fokker--Planck formulation, where a separate density \(p(\by,t;\bx_0,0)\) must be evolved for each initial point \(\bx_0\).  Moreover, \eqref{eq:Q_PDE} evolves only the second moment, not the full probability density.  In two dimensions this means solving three coupled scalar equations, and in three dimensions six coupled scalar equations.

At the continuous level, the equation preserves the covariance structure.  If \(\bQ(\bx,0)\) is symmetric positive semidefinite and \(\bD(\bx,t)\) is symmetric positive semidefinite, then \(\bQ(\bx,t)\) remains symmetric positive semidefinite along the flow.  In the present work, we do not build a full cone-preserving integrator.  Instead, we exploit the symmetry by evolving only the independent components of \(\bQ\), and we monitor positive semidefiniteness in the numerical experiments.

\subsubsection{Recovering the covariance indexed by initial position}

The vFTLE is a diagnostic defined over initial locations.  Therefore, after solving \eqref{eq:Q_PDE}, the covariance field must be pulled back from the arrival coordinate to the initial coordinate.  For an initial point \(\bx_0\), the deterministic arrival point at time \(T\) is \(\flow{0}{T}(\bx_0)\).  Hence the covariance associated with \(\bx_0\) is
\begin{equation}
    \bC_T(\bx_0)=\bQ(\flow{0}{T}(\bx_0),T).
    \label{eq:initial_indexed_cov}
\end{equation}
This is the covariance that approximates the second central moment of the stochastic arrival distribution associated with particles released from \(\bx_0\).

This pullback step is where the existing Eulerian flow-map machinery enters the computation.  Once the forward flow map \(\flow{0}{T}\) is available on the initial grid, the value \(\bQ(\flow{0}{T}(\bx_0),T)\) can be obtained by interpolation from the Eulerian covariance field at the final time.  The resulting covariance is then used to define the moment-based vFTLE in the next section.

It is useful to emphasize the distinction between the two coordinate descriptions.  The field \(\bQ(\bx,T)\) is stored as a function of the arrival coordinate \(\bx\), while \(\bC_T(\bx_0)\) is the same covariance information indexed by the take-off coordinate \(\bx_0\).  The two are related by the deterministic flow map.  This distinction is important because coherent-structure diagnostics are usually plotted over the initial grid, whereas the Eulerian covariance equation is most naturally solved over the physical grid at the current time.

\begin{remark}[Forward and backward viewpoints]
One may alternatively define a covariance field directly over the initial coordinate by setting \(\widetilde{\bQ}(\bx_0,t)=\bQ(\flow{0}{t}(\bx_0),t)\).  This eliminates the final interpolation in \eqref{eq:initial_indexed_cov}, but the evolution is then written in the initial coordinate and loses the standard Eulerian transport form.  The formulation \eqref{eq:Q_PDE} is preferable for implementation because it is posed on the physical grid and can use existing Eulerian advection and flow-map solvers.
\end{remark}

\begin{remark}[Boundary conditions]
The covariance transport equation requires boundary treatment consistent with the deterministic transport problem.  For periodic benchmark flows, periodic boundary conditions are imposed on all components of \(\bQ\).  For bounded domains with inflow boundaries, one must prescribe covariance data on inflow portions of the boundary.  In the simplest case of deterministic particle release with no pre-existing uncertainty, this inflow covariance is set to zero.  More complicated boundary conditions may be needed when modeling uncertainty entering from outside the computational domain.
\end{remark}

\subsubsection{Component form in two dimensions}

We now write \eqref{eq:Q_PDE} in component form for \(d=2\).  Let \(\bu=(u_1,u_2)^T\), and denote the velocity gradient by
\[
    \nabla_{\bx}\bu
    =
    \begin{pmatrix}
    a & b\\
    c & d
    \end{pmatrix},
    \qquad
    a=\partial_{x_1}u_1,\quad
    b=\partial_{x_2}u_1,\quad
    c=\partial_{x_1}u_2,\quad
    d=\partial_{x_2}u_2 .
\]
Since the covariance tensor is symmetric, we write
\[
    \bQ=
    \begin{pmatrix}
    q_{11} & q_{12}\\
    q_{12} & q_{22}
    \end{pmatrix},
    \qquad
    \bD=
    \begin{pmatrix}
    D_{11} & D_{12}\\
    D_{12} & D_{22}
    \end{pmatrix}.
\]
Then \eqref{eq:Q_PDE} is equivalent to the following three scalar transport--reaction equations:
\begin{equation}
\begin{aligned}
    \partial_t q_{11}+\bu\cdot\nabla q_{11}
    &=2a q_{11}+2b q_{12}+D_{11},\\
    \partial_t q_{12}+\bu\cdot\nabla q_{12}
    &=c q_{11}+(a+d)q_{12}+b q_{22}+D_{12},\\
    \partial_t q_{22}+\bu\cdot\nabla q_{22}
    &=2c q_{12}+2d q_{22}+D_{22}.
\end{aligned}
\label{eq:component_2d}
\end{equation}
This scalar formulation is the one used in the numerical implementation.  It preserves symmetry by construction, since \(q_{21}\) is never evolved separately.  Positive semidefiniteness corresponds to the inequalities \(q_{11}\geq0\), \(q_{22}\geq0\), and \(q_{11}q_{22}-q_{12}^2\geq0\).  These inequalities are preserved by the continuous covariance equation, but not necessarily by a generic componentwise discretization.  Therefore, in the numerical experiments we check the minimum eigenvalue of \(\bQ\) as a diagnostic of whether the discretization remains consistent with the covariance interpretation.

In three dimensions, the same idea gives six scalar equations for \(q_{11},q_{12},q_{13},q_{22},q_{23}\), and \(q_{33}\).  This remains far cheaper than evolving a full probability density for each initial point.  The reduction in dimension is especially significant in three-dimensional flows, where a Fokker--Planck solve is already expensive for a single initial point, while the covariance transport equation requires only a fixed number of scalar fields independent of the number of initial particles.

\subsection{Moment-based vFTLE}

We now define the flow diagnostic obtained from the Eulerian covariance transport equation.  Recall that the covariance field \(\bQ(\bx,t)\) is stored in the current physical coordinate, whereas coherent-structure diagnostics are usually plotted over the initial coordinate.  We therefore first pull back the covariance field by the deterministic flow map and use \(\bC_T(\bx_0)=\bQ(\flow{0}{T}(\bx_0),T)\) as the leading-order covariance associated with particles released from \(\bx_0\).

\begin{definition}[Moment-based vFTLE]
Let \(\bQ\) solve \eqref{eq:Q_PDE}, and let \(\bC_T(\bx_0)\) be defined by \eqref{eq:initial_indexed_cov}.  The moment-based variance FTLE, or moment-vFTLE, is defined by
\begin{equation}
    \mu_0^T(\bx_0)
    =
    \frac{1}{|T|}
    \log
    \sqrt{
    \lmax\!\left(\varepsilon^2\bC_T(\bx_0)\right)
    } .
    \label{eq:mvftle_def}
\end{equation}
When the global factor \(\varepsilon^2\) is omitted, we define the normalized moment-vFTLE by
\begin{equation}
    \widehat{\mu}_0^T(\bx_0)
    =
    \frac{1}{|T|}
    \log
    \sqrt{
    \lmax\!\left(\bC_T(\bx_0)\right)
    } .
    \label{eq:normalized_mvftle_def}
\end{equation}
\end{definition}

If \(\varepsilon\) is a constant scalar multiplying the stochastic forcing, then \(\mu_0^T\) and \(\widehat{\mu}_0^T\) differ only by the additive constant \(|T|^{-1}\log\varepsilon\).  Hence they have the same ridge locations.  The normalized quantity \(\widehat{\mu}_0^T\) is often more convenient for visualization because it removes the trivial dependence on the global noise amplitude.  However, if the strength of the uncertainty is spatially or temporally varying and is already incorporated into \(\bD(\bx,t)\), then the diagnostic should be computed directly from the covariance generated by that diffusion tensor.

The interpretation of \(\bC_T(\bx_0)\) is different from that of the deterministic Cauchy--Green tensor.  The Cauchy--Green tensor measures the sensitivity of the deterministic arrival point with respect to changes in the initial condition.  By contrast, \(\bC_T(\bx_0)\) measures the leading-order spread of stochastic arrival locations around the deterministic trajectory starting from \(\bx_0\).  Thus \(\mu_0^T\) should be interpreted as a finite-time measure of uncertainty-induced spreading rather than as a deterministic stretching exponent.

The tensor formulation also retains directional information.  If \(\be_{\max}(\bx_0)\) is a unit eigenvector associated with \(\lmax(\bC_T(\bx_0))\), then the pair \((\lmax(\bC_T(\bx_0)),\be_{\max}(\bx_0))\) gives the dominant magnitude and direction of stochastic spreading at the final time.  This direction is expressed in the arrival coordinate, since \(\bC_T(\bx_0)\) is the covariance of the stochastic arrival cloud.  Such directional information is unavailable if one reduces the computation to a scalar indicator at the outset.

\subsection{Connection with deterministic deformation}\label{sec:deformation_connection}

The covariance formulation also provides a useful bridge between deterministic FTLE and stochastic spreading.  Deterministic FTLE measures the amplification of infinitesimal perturbations placed at the initial time, while the process-noise covariance in \eqref{eq:cov_representation} measures the cumulative amplification of uncertainty injected continuously along the trajectory.  Both effects are described by the same covariance equation, but with different sources of uncertainty.

\begin{proposition}[Recovery of deterministic FTLE from initial covariance]\label{prop:ftle_initial_covariance}
Let \(\bD\equiv \bzero\), and suppose that the initial particle location has isotropic covariance \(\bC(0;\bx_0)=\alpha\bI\), with \(\alpha>0\).  Let \(\bF(T,0;\bx_0)=\nabla_{\bx_0}\flow{0}{T}(\bx_0)\) be the deformation gradient of the deterministic flow.  Then
\begin{equation}
    \bC_T(\bx_0)
    =
    \alpha\,\bF(T,0;\bx_0)\bF(T,0;\bx_0)^T .
    \label{eq:initial_covariance_recovery}
\end{equation}
Consequently, the normalized moment-vFTLE defined from this initial covariance satisfies
\begin{equation}
    \widehat{\mu}_0^T(\bx_0)
    =
    \sigma_0^T(\bx_0)
    +
    \frac{1}{|T|}\log\sqrt{\alpha},
    \label{eq:ftle_recovery_constant}
\end{equation}
where \(\sigma_0^T\) is the deterministic FTLE.
\end{proposition}

\begin{proof}
When \(\bD\equiv\bzero\), the representation formula with nonzero initial covariance gives \(\bC_T=\bF(T,0)\bC(0)\bF(T,0)^T\).  Substituting \(\bC(0)=\alpha\bI\) gives \eqref{eq:initial_covariance_recovery}.  The matrices \(\bF\bF^T\) and \(\bF^T\bF\) have the same nonzero eigenvalues, and \(\bF^T\bF\) is the deterministic Cauchy--Green tensor.  Hence \(\lmax(\bC_T)=\alpha\lmax(\bF^T\bF)\), and \eqref{eq:ftle_recovery_constant} follows from the definitions of \(\widehat{\mu}_0^T\) and \(\sigma_0^T\).
\end{proof}

This result shows that the covariance formulation contains deterministic FTLE as a special case.  If uncertainty is introduced only as an isotropic perturbation in the initial condition, the covariance diagnostic gives the same ridge structure as deterministic FTLE.  By contrast, the process-noise setting considered in the main part of this paper uses \(\bC(0)=\bzero\) and \(\bD\neq\bzero\), so uncertainty is injected throughout the time interval rather than only at the initial time.

For process noise, the covariance at time \(T\) can be written as
\begin{equation}
    \bC_T(\bx_0)
    =
    \int_0^T
    \bF(T,s;\bx_0)\,\bD(\bx(s),s)\,\bF(T,s;\bx_0)^T\,ds .
    \label{eq:process_noise_integrated_deformation}
\end{equation}
Thus the final covariance is a time-integrated deformation tensor: uncertainty injected at time \(s\) is propagated by the deterministic deformation from \(s\) to \(T\).  If the diffusion tensor is uniformly bounded in the Loewner order along the trajectory, namely
\[
    d_-\bI \preceq \bD(\bx(t),t) \preceq d_+\bI,
    \qquad 0\leq t\leq T,
\]
with \(0\leq d_-\leq d_+\), then \eqref{eq:process_noise_integrated_deformation} implies
\begin{equation}
    d_-\int_0^T \bF(T,s)\bF(T,s)^T\,ds
    \preceq
    \bC_T(\bx_0)
    \preceq
    d_+\int_0^T \bF(T,s)\bF(T,s)^T\,ds .
    \label{eq:deformation_integral_bounds}
\end{equation}
These bounds explain why the process-noise moment-vFTLE is often strongly related to deterministic stretching, while not being identical to deterministic FTLE.  Deterministic FTLE depends on the single deformation gradient from \(0\) to \(T\), whereas \eqref{eq:process_noise_integrated_deformation} accumulates deformations from all injection times \(s\) to the final time.

\begin{proposition}[Monotonicity with respect to diffusion]\label{prop:diffusion_monotonicity}
Let \(\bC_T^{(1)}(\bx_0)\) and \(\bC_T^{(2)}(\bx_0)\) be the covariance matrices generated by two diffusion tensors \(\bD_1\) and \(\bD_2\) along the same deterministic trajectory, with zero initial covariance.  If \(\bD_1(\bx(t),t)\preceq \bD_2(\bx(t),t)\) for \(0\leq t\leq T\), then
\[
    \bC_T^{(1)}(\bx_0)\preceq \bC_T^{(2)}(\bx_0),
    \qquad
    \lmax\!\left(\bC_T^{(1)}(\bx_0)\right)
    \leq
    \lmax\!\left(\bC_T^{(2)}(\bx_0)\right).
\]
\end{proposition}

\begin{proof}
Subtracting the two representation formulas gives
\[
    \bC_T^{(2)}-\bC_T^{(1)}
    =
    \int_0^T
    \bF(T,s)\left(\bD_2(\bx(s),s)-\bD_1(\bx(s),s)\right)\bF(T,s)^T\,ds .
\]
The integrand is positive semidefinite for every \(s\), hence the integral is positive semidefinite.  This gives the Loewner-order inequality.  The largest-eigenvalue inequality follows from monotonicity of \(\lmax\) on symmetric matrices under the Loewner order.
\end{proof}

The monotonicity result provides a basic consistency check: increasing the diffusion covariance cannot decrease the leading-order stochastic spreading.  Together with Proposition~\ref{prop:ftle_initial_covariance}, it also clarifies the role of the diffusion tensor in shaping the moment-vFTLE field.

\subsection{Relation with full Fokker--Planck vFTLE}

The moment-vFTLE is designed as a leading-order approximation to the PDF-based vFTLE in the small-noise regime.  The following result makes this connection precise at the covariance level.  The statement is local in \(\bx_0\) and assumes that the deterministic trajectory does not interact with the boundary during the time interval under consideration.

\begin{assumption}\label{ass:smoothness}
Assume that \(\bu\) and \(\bB\) are sufficiently smooth in \(\bx\) and continuous in time, with the derivatives required below uniformly bounded on the relevant part of \(\Omega\times[0,T]\).  Assume also that the deterministic trajectory \(\flow{0}{t}(\bx_0)\) remains in the interior of \(\Omega\) for \(0\leq t\leq T\).
\end{assumption}

\begin{theorem}[Leading-order covariance consistency]\label{thm:cov_consistency}
Under Assumption~\ref{ass:smoothness}, let \(\bSigma_0^T(\bx_0)\) denote the covariance matrix of the stochastic arrival location \(\bX_T\) generated by \eqref{eq:sde}.  Let \(\bC_T(\bx_0)\) be the covariance obtained from the linearized stochastic displacement, equivalently from \eqref{eq:cov_ode} or from the Eulerian pullback formula \eqref{eq:initial_indexed_cov}.  Then
\begin{equation}
    \bSigma_0^T(\bx_0)
    =
    \varepsilon^2\bC_T(\bx_0)+O(\varepsilon^3),
    \label{eq:cov_consistency}
\end{equation}
where the remainder is locally uniform in \(\bx_0\).  The constant in the \(O(\varepsilon^3)\) term may depend on \(T\), the velocity field, the diffusion coefficient, and their derivatives.
\end{theorem}

\begin{proof}[Formal proof]
Let \(\bx(t)=\flow{0}{t}(\bx_0)\).  From the small-noise expansion \eqref{eq:small_noise_expansion}, write \(\bX_t=\bx(t)+\varepsilon\etaVec_t+\varepsilon^2\boldsymbol{\zeta}_t+O(\varepsilon^3)\).  Matching terms of order \(\varepsilon\) gives the linearized stochastic equation \eqref{eq:linearized_eta}.  Since \(\etaVec_0=\bzero\) and the Brownian increment has zero mean, \(\E[\etaVec_t]=\bzero\).  Hence \(\E[\bX_T]=\bx(T)+O(\varepsilon^2)\), and \(\bX_T-\E[\bX_T]=\varepsilon\etaVec_T+\varepsilon^2(\boldsymbol{\zeta}_T-\E[\boldsymbol{\zeta}_T])+O(\varepsilon^3)\).  Taking the covariance gives \(\Cov(\bX_T)=\varepsilon^2\E[\etaVec_T\etaVec_T^T]+O(\varepsilon^3)\).  Finally, applying It\^o's product rule to \(\etaVec_t\etaVec_t^T\) gives the covariance equation \eqref{eq:cov_ode}, so \(\E[\etaVec_T\etaVec_T^T]=\bC_T(\bx_0)\).  This proves the stated leading-order relation formally.  A rigorous proof follows from standard SDE perturbation estimates and moment bounds under the smoothness assumptions above.
\end{proof}

\begin{corollary}[Consistency of moment-vFTLE]\label{cor:mvftle_consistency}
Assume the hypotheses of Theorem~\ref{thm:cov_consistency}.  Suppose in addition that \(\lmax(\bC_T(\bx_0))\geq c>0\) on the region of interest.  Then the PDF-based vFTLE \(v_0^T\) and the moment-based vFTLE \(\mu_0^T\) satisfy \(v_0^T(\bx_0)-\mu_0^T(\bx_0)=O(\varepsilon)\), locally uniformly in \(\bx_0\).
\end{corollary}

\begin{proof}
By Theorem~\ref{thm:cov_consistency}, \(\bSigma_0^T(\bx_0)=\varepsilon^2\bC_T(\bx_0)+O(\varepsilon^3)\).  Since the largest eigenvalue is Lipschitz continuous with respect to symmetric matrix perturbations, \(\lmax(\bSigma_0^T)=\varepsilon^2\lmax(\bC_T)+O(\varepsilon^3)\).  The lower bound on \(\lmax(\bC_T)\) ensures that the logarithm is evaluated away from zero.  Applying the map \(s\mapsto |T|^{-1}\log\sqrt{s}\) gives the result.
\end{proof}

\begin{remark}
The lower bound \(\lmax(\bC_T)\geq c>0\) is used only to avoid degeneracy in the logarithm.  At points where the leading-order covariance is zero or extremely small, the vFTLE becomes sensitive to higher-order terms and numerical regularization.  In computations, one may replace \(\lmax(\bC_T)\) by \(\lmax(\bC_T)+\delta\), with a small positive \(\delta\), when a regularized diagnostic is desired.
\end{remark}

\begin{remark}
For a spatially constant noise amplitude \(\varepsilon\), the logarithm in vFTLE introduces the additive constant \(|T|^{-1}\log\varepsilon\).  Thus the normalized moment-vFTLE \(\widehat{\mu}_0^T\) has the same ridge locations as \(\mu_0^T\) in this case.  When comparing different noise amplitudes, or when the magnitude of the uncertainty is incorporated into a spatially varying diffusion tensor, one should state explicitly whether the normalized or unnormalized diagnostic is being used.
\end{remark}

\subsection{Relation with stochastic sensitivity}

Stochastic sensitivity also describes the leading-order variance generated by small stochastic perturbations.  In the present notation, this information is contained in the covariance tensor \(\bC_T(\bx_0)\).  For any unit vector \(\bn\in\R^d\), the variance of the projected stochastic displacement in the direction \(\bn\) is \(\Var(\bn\cdot\etaVec_T)=\bn^T\bC_T(\bx_0)\bn\).  Maximizing over all unit directions gives the Rayleigh quotient identity
\[
    \sup_{\|\bn\|=1}\Var(\bn\cdot\etaVec_T)
    =
    \sup_{\|\bn\|=1}\bn^T\bC_T(\bx_0)\bn
    =
    \lmax(\bC_T(\bx_0)).
\]
Thus the scalar stochastic sensitivity and the largest eigenvalue of the covariance tensor contain closely related leading-order information.

The distinction is mainly methodological and interpretive.  In earlier vFTLE computations, stochastic sensitivity may be used as an efficient indicator to identify regions where the more expensive PDF-based vFTLE should be computed.  In the present approach, the covariance tensor is not merely a screening indicator.  It is computed directly through the Eulerian covariance transport equation and then used as the main approximation to the covariance entering vFTLE.  This also allows one to retain the full tensor information, including dominant spreading directions, and to handle anisotropic diffusion tensors naturally.

\begin{proposition}[Scalar sensitivity from the covariance tensor]
Let \(\bC_T(\bx_0)\) be the covariance of the leading-order stochastic displacement \(\etaVec_T\).  Define the scalar sensitivity by \(S^2(\bx_0)=\sup_{\|\bn\|=1}\Var(\bn\cdot\etaVec_T)\).  Then \(S^2(\bx_0)=\lmax(\bC_T(\bx_0))\).  Consequently, the normalized moment-vFTLE can be written as
\[
    \widehat{\mu}_0^T(\bx_0)
    =
    \frac{1}{|T|}
    \log\sqrt{S^2(\bx_0)} .
\]
\end{proposition}

\begin{proof}
Since \(\bC_T(\bx_0)\) is symmetric positive semidefinite, the variance in direction \(\bn\) is \(\bn^T\bC_T(\bx_0)\bn\).  The maximum of this Rayleigh quotient over all unit vectors is the largest eigenvalue of \(\bC_T(\bx_0)\).  Substituting \(S^2(\bx_0)=\lmax(\bC_T(\bx_0))\) into \eqref{eq:normalized_mvftle_def} gives the stated expression.
\end{proof}

\begin{remark}
Although the scalar sensitivity \(S^2\) and the normalized moment-vFTLE contain the same largest-eigenvalue information, they emphasize different aspects of the computation.  The sensitivity \(S^2\) is a variance amplitude, while \(\widehat{\mu}_0^T\) rescales this amplitude logarithmically and by the integration time, matching the convention used in FTLE-type diagnostics.
\end{remark}

\section{Numerical Discretization and Computational Complexity}
\label{sec:numerical_method}

\subsection{Grid, notation, and splitting of the equation}

Let \(\Omega_h=\{\bx_i\}\) be a Cartesian grid and let \(t_n=n\Delta t\), \(n=0,\ldots,N_t\), with \(t_{N_t}=T\).  We write \(\bQ_i^n\approx\bQ(\bx_i,t_n)\).  The velocity \(\bu(\bx_i,t_n)\), velocity gradient \(\nabla\bu(\bx_i,t_n)\), and diffusion tensor \(\bD(\bx_i,t_n)\) are assumed to be available analytically or approximated numerically by interpolation and finite differences.

The covariance equation \eqref{eq:Q_PDE} consists of an advection part and a local source part.  The advection part is \(\partial_t\bQ+\bu\cdot\nabla\bQ=0\), and the source part is \(d\bQ/dt=(\nabla\bu)\bQ+\bQ(\nabla\bu)^T+\bD\).  This decomposition is useful because the advection step can reuse semi-Lagrangian interpolation and flow-map construction, while the source step is local at each grid point.  In the implementation used here, the tensor symmetry is enforced by evolving only the independent components of \(\bQ\).  In two dimensions these are \(q_{11},q_{12}\), and \(q_{22}\).

\subsection{Semi-Lagrangian transport step}

For a pure transport step from \(t_n\) to \(t_{n+1}\), the value at a grid point \(\bx_i\) at the new time is obtained from the departure point at the old time.  Let \(\flow{t_{n+1}}{t_n}(\bx_i)\) denote the short-time backward flow map.  Then the exact transport solution satisfies \(\bQ(\bx_i,t_{n+1})=\bQ(\flow{t_{n+1}}{t_n}(\bx_i),t_n)\).  We approximate this by
\begin{equation}
    \bQ_i^{\mathrm{adv}}
    =
    \mathcal{I}_h[\bQ^n]\!\left(\flow{t_{n+1}}{t_n}(\bx_i)\right),
    \label{eq:semi_lag_transport}
\end{equation}
where \(\mathcal{I}_h\) is a componentwise interpolation operator.  The short-time backward map can be computed using the same characteristic or Eulerian flow-map interpolation techniques used for deterministic FTLE computations.  For smooth solutions, higher-order interpolation can be used to reduce numerical diffusion; however, because interpolation is applied componentwise, positive semidefiniteness of the covariance tensor is not guaranteed exactly at the discrete level.

\subsection{Local source update}

After the transport step, the local source equation is advanced at each grid point.  Let \(\bA_i^{n+\frac12}\) and \(\bD_i^{n+\frac12}\) denote approximations of \(\nabla\bu(\bx_i,t_{n+\frac12})\) and \(\bD(\bx_i,t_{n+\frac12})\).  A simple second-order midpoint update is
\[
    \bQ_i^{n+1}
    =
    \bQ_i^{\mathrm{adv}}
    +
    \Delta t\left(
    \bA_i^{n+\frac12}\bQ_i^{\mathrm{mid}}
    +
    \bQ_i^{\mathrm{mid}}(\bA_i^{n+\frac12})^T
    +
    \bD_i^{n+\frac12}
    \right),
\]
where \(\bQ_i^{\mathrm{mid}}\) is a midpoint approximation.  In practice, \(\bQ_i^{\mathrm{mid}}\) can be obtained by a predictor step, or the source equation can be integrated using a standard second-order Runge--Kutta method.  Since the dimension \(d\) is small, this local source update is inexpensive compared with the advection and interpolation steps.

In the present work, we use this componentwise source discretization rather than a fully structure-preserving covariance integrator.  After each complete time step, we symmetrize the numerical tensor by replacing \(\bQ_i\) with \((\bQ_i+\bQ_i^T)/2\), although in the two-dimensional component formulation symmetry is already built in.  We also monitor the smallest eigenvalue of \(\bQ_i\) over the grid.  If small negative eigenvalues appear due to interpolation or time-discretization error, they can be clipped to zero for visualization; however, this correction is treated as a numerical safeguard rather than as part of the formal discretization.

\begin{remark}
The continuous covariance equation preserves positive semidefiniteness, but a componentwise Eulerian discretization does not necessarily preserve this property exactly.  One could design cone-preserving updates based on matrix exponentials or square-root variables.  We do not pursue that direction here, since the main objective of the present work is to test the moment reduction from PDF-based vFTLE to covariance transport.  Structure-preserving covariance discretizations are left as a possible future refinement.
\end{remark}

\subsection{Second-order splitting algorithm}

A second-order splitting over one time step may be written symbolically as
\[
    \bQ^{n+1}
    \approx
    \mathcal{S}_{\Delta t/2}^{n+1}
    \mathcal{T}_{\Delta t}^{n\to n+1}
    \mathcal{S}_{\Delta t/2}^{n}
    \bQ^n,
\]
where \(\mathcal{T}_{\Delta t}^{n\to n+1}\) denotes the semi-Lagrangian transport update and \(\mathcal{S}_{\Delta t/2}^{n}\), \(\mathcal{S}_{\Delta t/2}^{n+1}\) denote half-step source updates.  This mirrors the operator-splitting philosophy used in PDF-based Fokker--Planck computations, but the evolved object is now a small covariance tensor field rather than a full density for each initial condition.

\begin{algorithm}[t]
\caption{Moment-based Eulerian vFTLE computation}
\DontPrintSemicolon
\KwIn{Velocity field \(\bu\), diffusion tensor \(\bD\), grid \(\Omega_h\), time interval \([0,T]\), noise amplitude \(\varepsilon\).}
\KwOut{Moment-vFTLE field \(\mu_0^T\) and, optionally, dominant spreading directions.}
Initialize \(\bQ_i^0=\bzero\) at all grid points \(\bx_i\).\;
Compute or store the short-time flow maps required for semi-Lagrangian transport.\;
\For{\(n=0,\ldots,N_t-1\)}{
    Apply a half-step local source update using \(\nabla\bu\) and \(\bD\) near \(t_n\).\;
    Advect the independent components of \(\bQ\) by the semi-Lagrangian transport step.\;
    Apply a half-step local source update using \(\nabla\bu\) and \(\bD\) near \(t_{n+1}\).\;
    Symmetrize \(\bQ\), if necessary, and record the minimum eigenvalue over the grid.\;
}
Compute the forward flow map \(\flow{0}{T}(\bx_i)\) on the initial grid.\;
Interpolate \(\bQ(\cdot,T)\) at \(\flow{0}{T}(\bx_i)\) to obtain \(\bC_T(\bx_i)\).\;
Set \(\mu_0^T(\bx_i)=|T|^{-1}\log\sqrt{\lmax(\varepsilon^2\bC_T(\bx_i))}\).\;
Store the dominant eigenvector of \(\bC_T(\bx_i)\) if directional information is desired.\;
\end{algorithm}

\subsection{Computational complexity}

We briefly compare the computational scaling of the proposed moment method with that of a full PDF-based vFTLE computation.  Let \(M\) denote the total number of grid points in the physical domain and let \(K\) denote the number of time steps.  For simplicity, we assume that all scalar fields are discretized on the same grid and that the cost of one transport or diffusion update for one scalar field is proportional to \(M\), up to constants depending on the spatial dimension, interpolation order, and time-stepping scheme.

In the PDF-based formulation, one evolves a probability density \(p(\by,t;\bx_0,0)\) for each initial grid point \(\bx_0\).  Each density is itself a scalar field over the physical grid.  Therefore, if vFTLE is required at all \(M\) initial grid points, the computation involves \(M\) density solves, each of size \(M\), over \(K\) time steps.  The resulting cost scales as \(O(KM^2)\), up to constants associated with the advection--diffusion solver and the computation of the first and second moments.  The memory requirement can be reduced by processing the initial points one at a time, but the total work remains quadratic in \(M\).

By contrast, the moment method evolves only the independent components of the covariance tensor field \(\bQ(\bx,t)\).  Since \(\bQ\) is symmetric, the number of scalar fields is \(d(d+1)/2\).  The covariance-transport part of the computation therefore scales as
\[
    O\!\left(KM\,\frac{d(d+1)}{2}\right).
\]
There is also the cost of computing the deterministic flow map and interpolating \(\bQ(\cdot,T)\) at the arrival locations \(\flow{0}{T}(\bx_0)\).  These operations scale linearly in \(M\) per time step for standard Eulerian or semi-Lagrangian implementations.  Hence the overall scaling of the moment method remains essentially linear in the number of grid points, rather than quadratic.

In two dimensions, the covariance tensor has only three independent components, \(q_{11},q_{12}\), and \(q_{22}\).  In three dimensions, it has six independent components.  Thus the number of evolved scalar fields is fixed by the spatial dimension and does not grow with the number of initial points at which the diagnostic is evaluated.  This is the main computational advantage of the proposed method: it replaces a family of Fokker--Planck solves indexed by \(\bx_0\) with a single Eulerian tensor-transport solve.

The comparison is summarized schematically as follows.  A full PDF-based vFTLE computation evolves \(M\) density fields, each with \(M\) degrees of freedom, giving \(O(KM^2)\) work.  The moment method evolves \(d(d+1)/2\) covariance fields, each with \(M\) degrees of freedom, giving \(O(KM d^2)\) work.  The reduction is especially significant when the diagnostic is required over a dense set of initial conditions.  The trade-off is that the moment method computes only the leading-order covariance in the small-noise regime, rather than the full probability density of stochastic arrivals.

\section{Numerical Experiments}
\label{sec:numerical_experiments}

We present three numerical experiments to assess the proposed moment-based vFTLE formulation.  The first is a linear verification problem with an analytic covariance solution; it isolates the covariance component equations, source update, and eigenvalue computation from flow-map errors.  The second is the standard double-gyre benchmark with isotropic diffusion, where the normalized moment-vFTLE is compared with deterministic FTLE and additional covariance diagnostics are extracted.  The third uses the same double-gyre velocity field with spatially and temporally varying anisotropic diffusion, illustrating the directional information retained by the covariance tensor.  Throughout the experiments, we monitor the smallest covariance eigenvalue as a check that the computed tensor remains consistent with its probabilistic interpretation.

\begin{figure}[!htb]
    \centering
    \includegraphics[width=\textwidth]{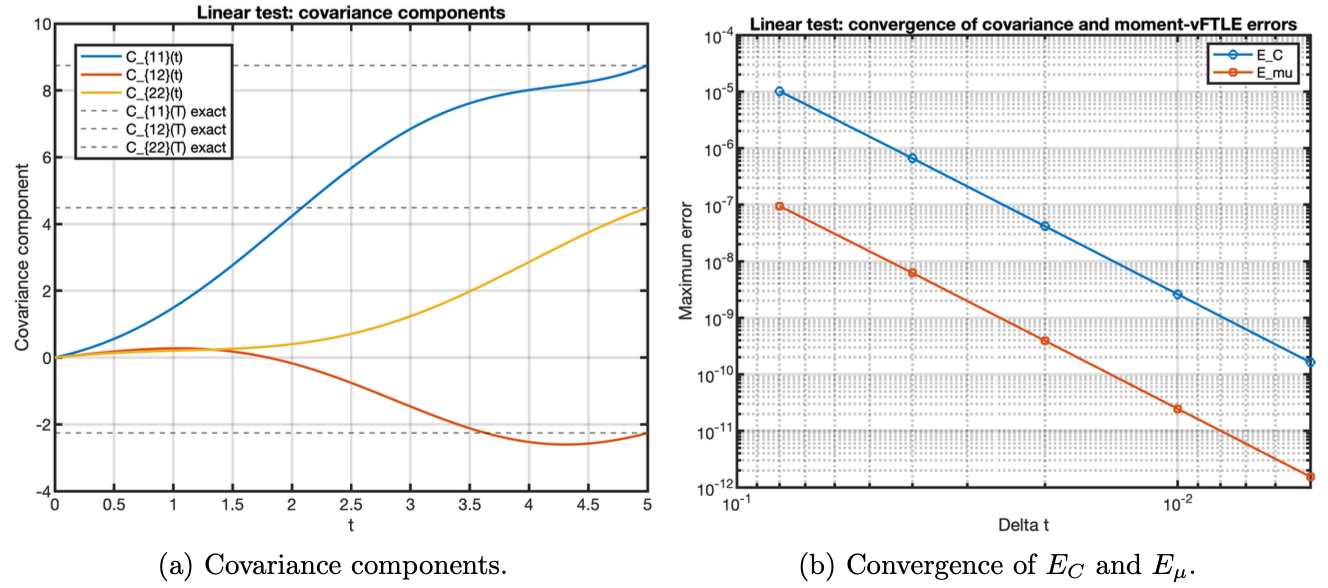}
    \caption{Linear verification test.  The covariance equation is exact for the linear velocity field with additive diffusion.  Panel (a) shows that the numerical covariance components agree with the exact final covariance values.  Panel (b) shows fourth-order convergence of both the covariance error and the moment-vFTLE error under time-step refinement.}
    \label{fig:linear_verification}
\end{figure}

\begin{figure}[!htb]
    \centering
    \includegraphics[width=\textwidth]{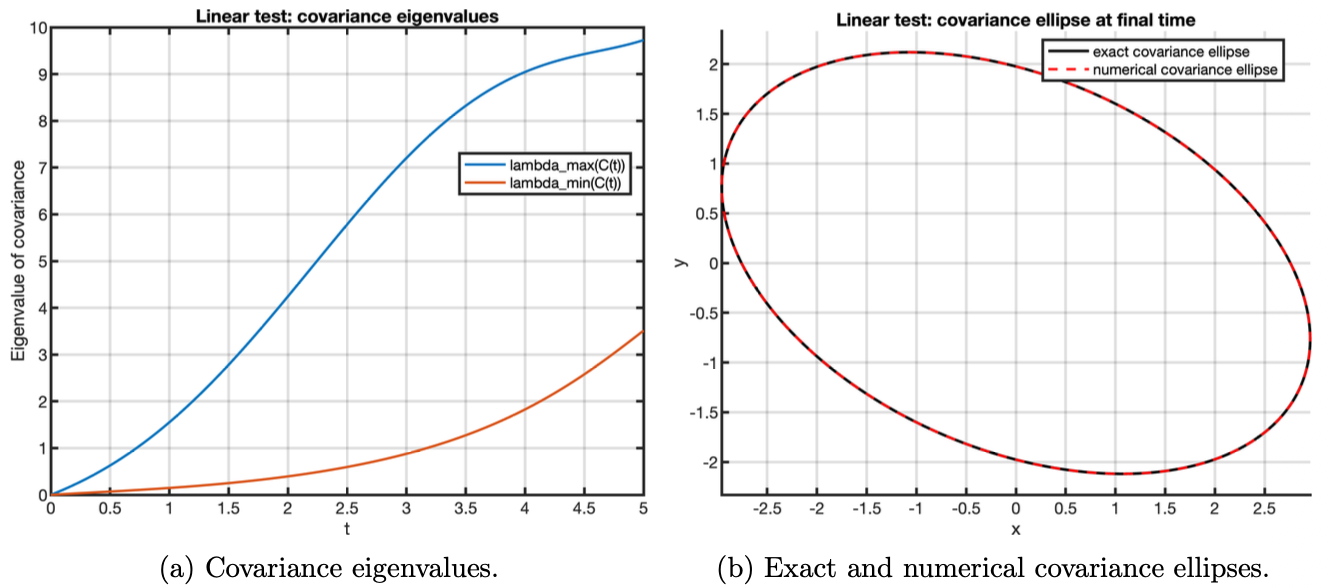}
    \caption{Additional diagnostics for the linear verification test.  The eigenvalue plot confirms positivity of the covariance matrix during the computation.  The covariance ellipse comparison verifies that the numerical tensor reproduces both the magnitude and orientation of the exact final covariance.}
    \label{fig:linear_tensor_diagnostics}
\end{figure}

\subsection{Example 1: Linear benchmark with analytic covariance}

We first consider a linear velocity field for which the covariance equation is exact and can be solved independently of any flow-map approximation.  This example verifies the component form of the covariance equation, the local source update, and the largest-eigenvalue calculation used in the moment-vFTLE diagnostic.

The velocity field is \(\bu(\bx,t)=\bA\bx\), where
\[
    \bA=
    \begin{pmatrix}
        0.30 & 1.00\\
       -0.40 & -0.10
    \end{pmatrix}.
\]
We use a constant anisotropic diffusion tensor \(\bD=\bR\operatorname{diag}(1,0.15)\bR^T\), where \(\bR\) is the rotation matrix with angle \(\pi/6\).  Since both \(\bA\) and \(\bD\) are independent of space and time, the leading-order covariance is spatially constant and satisfies \(d\bC/dt=\bA\bC+\bC\bA^T+\bD\), with \(\bC(0)=\bzero\).  This is the exact covariance equation for the corresponding linear stochastic system with additive diffusion.  Therefore, any numerical error in this test comes from the covariance source update and the eigenvalue computation, not from flow-map interpolation.

The exact covariance at time \(T\) is
\[
    \bC_{\rm exact}(T)
    =
    \int_0^T e^{(T-s)\bA}\bD e^{(T-s)\bA^T}\,ds.
\]
Equivalently, after vectorization, \(\operatorname{vec}(\bC)\) satisfies a linear inhomogeneous ODE with coefficient matrix \(\bI\otimes\bA+\bA\otimes\bI\), so the reference solution can be computed accurately by a matrix exponential.  The numerical solution evolves the three independent components \(C_{11}\), \(C_{12}\), and \(C_{22}\) using the scalar component equations derived from \(d\bC/dt=\bA\bC+\bC\bA^T+\bD\).  A fourth-order Runge--Kutta method is used for this local source equation.

We take \(T=5\), use a \(101\times101\) grid, and compare the numerical covariance with the exact covariance.  Although the grid is not needed for the exact solution, it allows us to use the same diagnostic notation as in the nonlinear examples.  The errors are measured by \(E_C=\max_{\bx_i}\|\bC_{\rm num}(T;\bx_i)-\bC_{\rm exact}(T)\|_F\) and \(E_\mu=\max_{\bx_i}|\mu_{\rm num}(\bx_i)-\mu_{\rm exact}|\).  For \(\Delta t=0.01\), the exact covariance is
\[
    \bC_{\rm exact}(T)
    =
    \begin{pmatrix}
        8.750228 & -2.261767\\
       -2.261767 &  4.490160
    \end{pmatrix}.
\]
The corresponding eigenvalues are \(\lambda_{\max}\approx9.727063\) and \(\lambda_{\min}\approx3.513325\), giving an anisotropy ratio of approximately \(2.77\).  The numerical errors are \(E_C\approx2.60\times10^{-9}\) and \(E_\mu\approx2.46\times10^{-11}\).  The smallest eigenvalue of the computed covariance over the grid is approximately \(3.51\), confirming that the covariance remains positive definite in this test.

Figure~\ref{fig:linear_verification} summarizes the verification result.  The left panel shows the time evolution of the three covariance components, with dashed horizontal lines indicating the exact final values at \(T=5\).  The right panel shows the convergence of \(E_C\) and \(E_\mu\) under time-step refinement.  The observed convergence rates are approximately fourth order: from \(\Delta t=0.08\) to \(0.005\), the estimated rates for \(E_C\) are \(3.94\), \(3.99\), \(4.00\), and \(4.00\), and the corresponding rates for \(E_\mu\) are \(3.93\), \(3.99\), \(3.99\), and \(4.00\).  This is consistent with the fourth-order Runge--Kutta update and confirms that the component covariance equations and the subsequent largest-eigenvalue computation are implemented correctly.

Figure~\ref{fig:linear_tensor_diagnostics} shows two additional tensor diagnostics.  The covariance eigenvalues remain nonnegative throughout the integration, and the numerical covariance ellipse at \(T=5\) is visually indistinguishable from the exact ellipse.  In this example the moment-vFTLE field is spatially constant, because the velocity gradient and diffusion tensor are spatially independent.  This provides an additional qualitative check: the method does not introduce artificial spatial variation when the exact diagnostic is constant.

\begin{figure}[!htb]
    \centering
        \includegraphics[width=\textwidth]{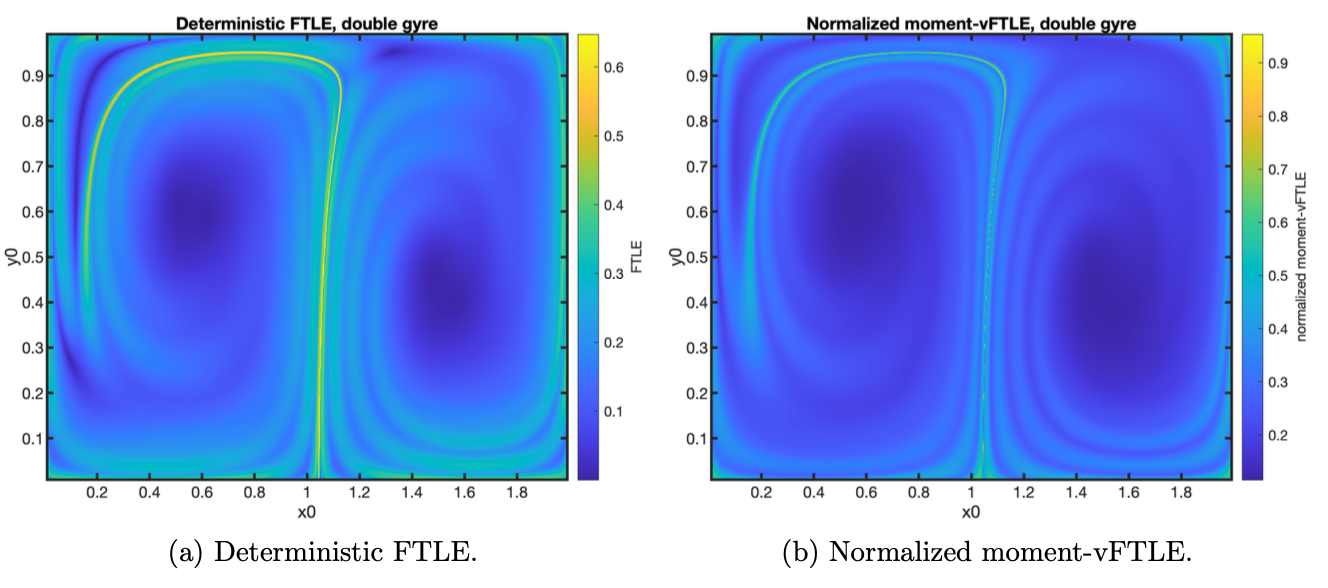}
    \caption{Double-gyre example with \(T=10\), \(\Delta t=0.02\), and a \(1024\times512\) grid.  The deterministic FTLE and the normalized moment-vFTLE reveal closely related coherent structures.  The moment-vFTLE is computed from the leading-order covariance tensor and therefore measures uncertainty-induced spreading rather than deterministic sensitivity to initial conditions.}
    \label{fig:double_gyre_fields}
\end{figure}

\begin{figure}[!htb]
    \centering
\includegraphics[width=\textwidth]{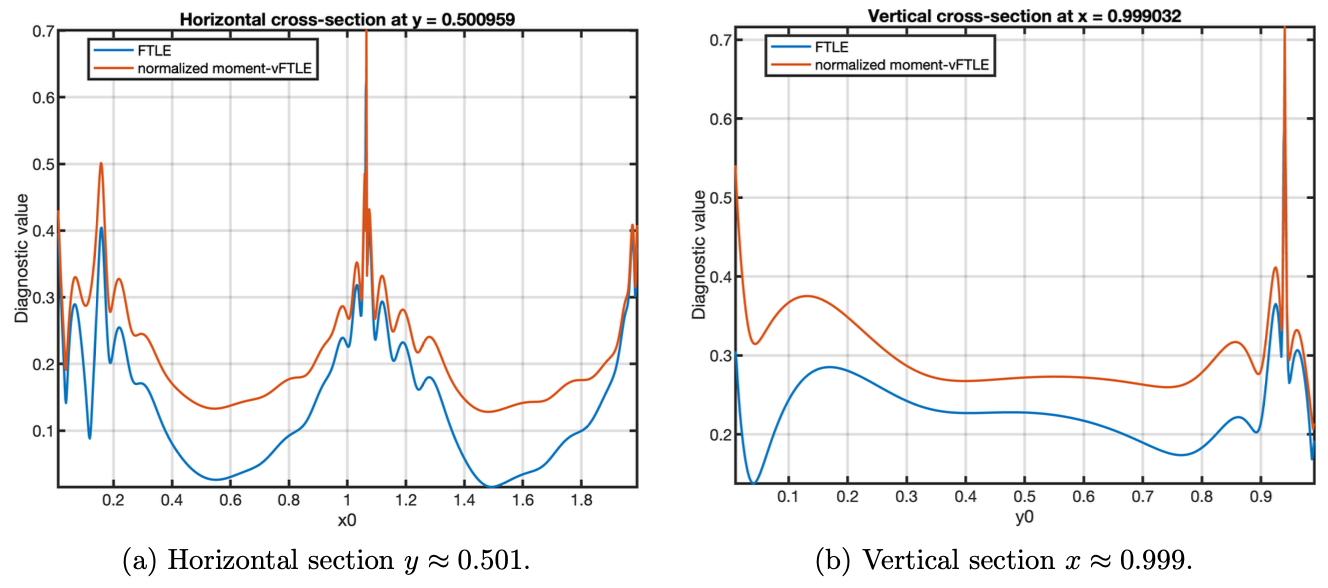}
    \caption{Cross-section comparison between deterministic FTLE and normalized moment-vFTLE.  The main peaks occur at essentially the same locations in these representative sections, while the profiles differ because the two quantities measure different forms of finite-time spreading.}
    \label{fig:double_gyre_sections}
\end{figure}

\begin{figure}[!htb]
    \centering
\includegraphics[width=\textwidth]{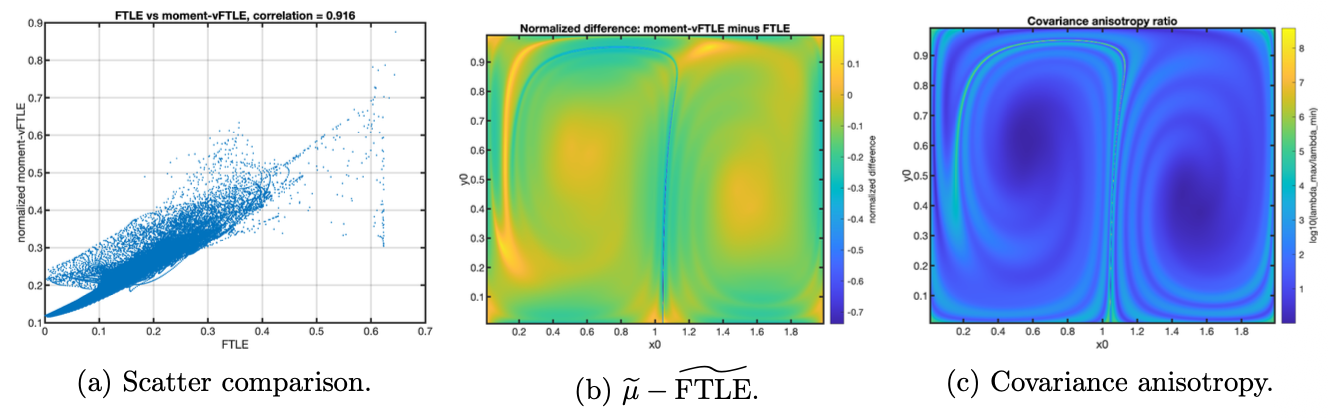}
    \caption{Additional diagnostics for the isotropic double-gyre example.  Panel (a) compares deterministic FTLE and normalized moment-vFTLE over a downsampled grid.  Panel (b) shows the difference between the two normalized scalar fields.  Panel (c) shows the anisotropy ratio of the final covariance tensor.  These diagnostics illustrate that the moment formulation gives access to tensor information beyond the scalar moment-vFTLE value.}
    \label{fig:double_gyre_additional_diagnostics}
\end{figure}

\subsection{Example 2: Double-gyre flow with isotropic diffusion}

We next consider the time-periodic double-gyre flow on the rectangular domain \(\Omega=[0.01,1.99]\times[0.01,0.99]\).  The velocity field is generated by the stream function \(\psi(x,y,t)=A\sin(\pi f(x,t))\sin(\pi y)\), where \(f(x,t)=a(t)x^2+b(t)x\), \(a(t)=\epsilon_g\sin(\omega t)\), and \(b(t)=1-2\epsilon_g\sin(\omega t)\).  The velocity is \(\bu=(-\partial_y\psi,\partial_x\psi)^T\).  In the computation we use \(A=0.1\), \(\epsilon_g=0.1\), and \(\omega=2\pi/10\), which are standard parameter values for this benchmark problem.

We take the integration time to be \(T=10\), use time step \(\Delta t=0.02\), and discretize the initial domain by a \(1024\times512\) grid.  For this validation test, we use the characteristic form of the covariance equation rather than the fully Eulerian semi-Lagrangian discretization.  That is, for each initial point \(\bx_0\), we integrate the deterministic trajectory together with
\[
    \frac{d\bC}{dt}
    =
    \bA(t)\bC+\bC\bA(t)^T+\bD(\bx(t),t),
    \qquad
    \bC(0;\bx_0)=\bzero,
\]
where \(\bA(t)=\nabla\bu(\bx(t),t)\).  This is equivalent to solving the Eulerian covariance transport equation along characteristics, but it avoids interpolation errors and therefore provides a clean verification of the moment model.  We use an isotropic diffusion tensor \(\bD=\bI\) in the leading-order covariance equation.  Since the global stochastic amplitude \(\varepsilon\) contributes only the additive constant \(|T|^{-1}\log\varepsilon\) to the vFTLE when \(\varepsilon\) is spatially constant, we plot the normalized moment-vFTLE \(\widehat{\mu}_0^T\).

The deterministic FTLE field and the normalized moment-vFTLE field are shown in Figure~\ref{fig:double_gyre_fields}.  The deterministic FTLE highlights the familiar ridge structures associated with finite-time stretching in the double-gyre flow.  The normalized moment-vFTLE produces a closely related ridge pattern, but it should not be expected to be identical to the deterministic FTLE.  Proposition~\ref{prop:ftle_initial_covariance} shows that deterministic FTLE would be recovered, up to an additive constant, if isotropic uncertainty were placed only in the initial condition.  Here, however, the covariance is generated by continuous process noise, so \eqref{eq:process_noise_integrated_deformation} shows that the diagnostic accumulates deformation from all injection times along the trajectory.  This explains why the two fields are strongly correlated while still retaining visible differences in magnitude and fine-scale structure.

The numerical ranges are summarized in Table~\ref{tab:double_gyre_summary}.  The deterministic FTLE ranges from approximately \(6.53\times10^{-7}\) to \(6.48\times10^{-1}\), while the normalized moment-vFTLE ranges from approximately \(1.15\times10^{-1}\) to \(9.54\times10^{-1}\).  On the downsampled grid used for diagnostic comparison, the Pearson correlation coefficient between the two scalar fields is approximately \(0.916\).  This high correlation confirms that the covariance-based diagnostic is strongly influenced by the same finite-time deformation mechanisms that generate deterministic FTLE ridges.  At the same time, the values are not identical because deterministic FTLE measures the amplification of initial perturbations, while the process-noise moment-vFTLE measures a time-integrated covariance response.  The minimum covariance eigenvalue is approximately \(4.97\times10^{-1}\), and the largest covariance eigenvalue reaches approximately \(1.93\times10^8\), reflecting the large range of uncertainty amplification over the domain.

\begin{table}[t]
    \centering
    \caption{Summary of the isotropic double-gyre computation.}
    \label{tab:double_gyre_summary}
    \begin{tabular}{ll}
        \hline
        Quantity & Value \\
        \hline
        Grid size & \(1024\times512\) \\
        Final time & \(T=10\) \\
        Time step & \(\Delta t=0.02\) \\
        Number of time steps & \(500\) \\
        Minimum deterministic FTLE & \(6.53\times10^{-7}\) \\
        Maximum deterministic FTLE & \(6.48\times10^{-1}\) \\
        Minimum normalized moment-vFTLE & \(1.15\times10^{-1}\) \\
        Maximum normalized moment-vFTLE & \(9.54\times10^{-1}\) \\
        Minimum \(\lambda_{\min}(\bC_T)\) & \(4.97\times10^{-1}\) \\
        Maximum \(\lambda_{\max}(\bC_T)\) & \(1.93\times10^{8}\) \\
        Maximum \(\log_{10}(\lambda_{\max}/\lambda_{\min})\) & \(8.58\) \\
        Downsampled FTLE--moment-vFTLE correlation & \(0.916\) \\
        \hline
    \end{tabular}
\end{table}

To compare the two diagnostics more directly, Figure~\ref{fig:double_gyre_sections} shows representative one-dimensional cross-sections.  Along the horizontal cross-section \(y\approx0.501\), both the deterministic FTLE and the normalized moment-vFTLE attain their maximum near \(x\approx1.065\).  Along the vertical cross-section \(x\approx0.999\), both diagnostics attain their maximum near \(y\approx0.940\).  The cross-section comparison reinforces the observation from Figure~\ref{fig:double_gyre_fields}: the moment-vFTLE identifies similar dynamically important regions, but its profile is smoother and reflects the accumulated stochastic covariance rather than only the deformation of initial perturbations.

Figure~\ref{fig:double_gyre_additional_diagnostics} gives additional information extracted from the same covariance tensor.  The scatter plot of deterministic FTLE versus normalized moment-vFTLE shows a strong positive relation, but also visible spread, consistent with the fact that the two diagnostics measure different finite-time effects.  The normalized difference field \(\widetilde{\mu}-\widetilde{\FTLE}\), where both fields are rescaled to \([0,1]\), highlights regions where covariance-based stochastic spreading is relatively stronger or weaker than deterministic stretching.  The covariance anisotropy field \(\log_{10}(\lambda_{\max}(\bC_T)/\lambda_{\min}(\bC_T))\) shows where the stochastic arrival covariance is highly elongated.  In the full-resolution computation, this quantity ranges from approximately \(9.62\times10^{-4}\) to \(8.58\), indicating that the stochastic arrival cloud can be nearly isotropic in some regions and extremely anisotropic in others.  The trace and area scale \(\sqrt{\det\bC_T}\) were also computed as total-variance and uncertainty-area diagnostics; we do not display them here to keep the main comparison focused on the largest eigenvalue and anisotropy.

This example suggests that the moment-based vFTLE captures coherent structures that are strongly related to deterministic finite-time stretching, while retaining a distinct uncertainty-based interpretation.  The diagnostic is obtained by evolving only the three independent components of the two-dimensional covariance tensor, rather than solving one Fokker--Planck equation for each initial point.  Thus the example illustrates both the computational motivation and the physical meaning of the proposed covariance moment approximation.

\begin{figure}[!htb]
    \centering
\includegraphics[width=\textwidth]{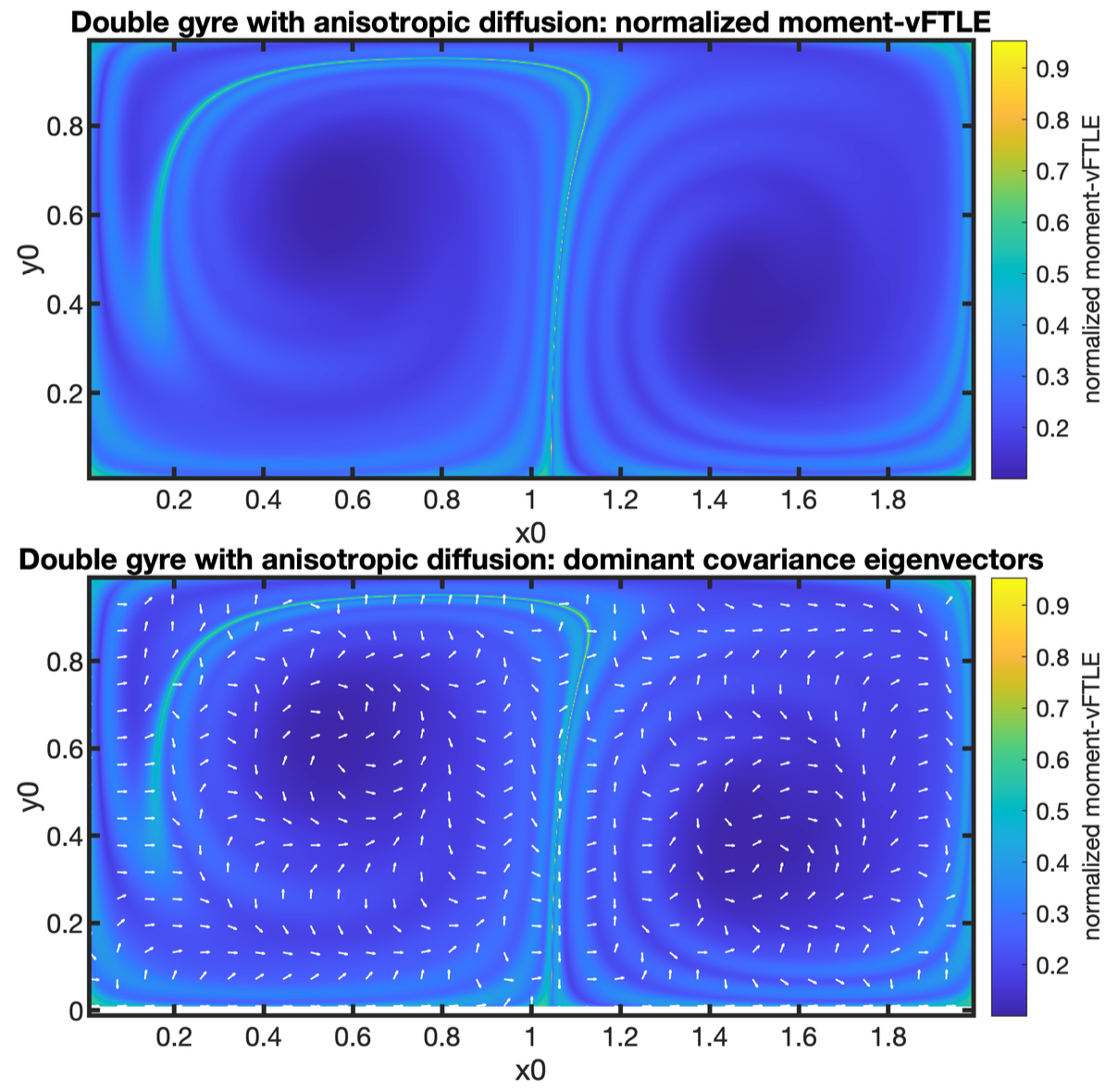}
    \caption{Double-gyre flow with spatially and temporally varying anisotropic diffusion.  Panel (a) shows the normalized moment-vFTLE computed from \(\lambda_{\max}(\bC_T)\).  Panel (b) overlays a subsampled plot of the dominant eigenvectors of \(\bC_T\) on the same scalar field.  The eigenvector directions reveal the orientation of the stochastic arrival covariance and show information that is not visible from the scalar diagnostic alone.}
    \label{fig:double_gyre_anisotropic}
\end{figure}

\begin{figure}[!htb]
    \centering
\includegraphics[width=\textwidth]{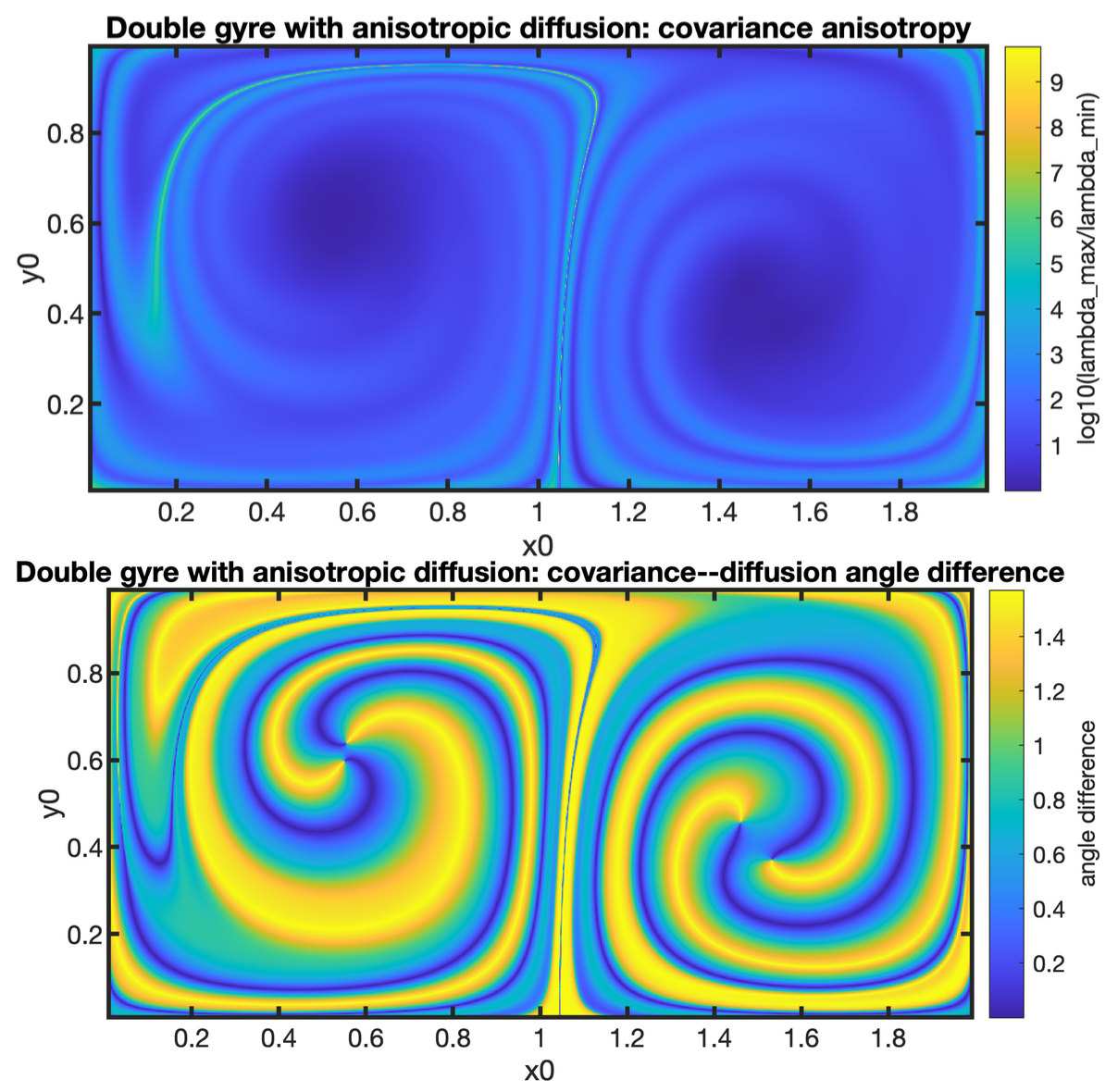}
    \caption{Additional tensor diagnostics for the anisotropic double-gyre example.  Panel (a) shows \(\log_{10}(\lambda_{\max}(\bC_T)/\lambda_{\min}(\bC_T))\), which measures the elongation of the stochastic arrival covariance.  Panel (b) shows the angle difference between the dominant covariance direction and the local dominant diffusion direction evaluated at the deterministic arrival location.  The difference between these directions demonstrates the cumulative influence of the flow on the final uncertainty orientation.}
    \label{fig:double_gyre_anisotropic_extra}
\end{figure}

\subsection{Example 3: Double-gyre flow with anisotropic spatially varying uncertainty}

We next consider the same double-gyre velocity field, but replace the isotropic diffusion tensor by a spatially and temporally varying anisotropic tensor.  This example highlights the tensor nature of the proposed method.  The scalar moment-vFTLE uses only \(\lambda_{\max}(\bC_T)\), but the full covariance tensor also provides the dominant direction of stochastic spreading through its eigenvectors.

The diffusion tensor is prescribed in the form \(\bD(\bx,t)=\bR(\theta(\bx,t))\operatorname{diag}(d_1(\bx,t),d_2(\bx,t))\bR(\theta(\bx,t))^T\), where \(\bR(\theta)\) is the two-dimensional rotation matrix.  In the computation we take \(d_1(\bx,t)=1+0.5\sin^2(\pi x)\sin^2(\pi y)\), \(d_2(\bx,t)=0.05+0.05\cos^2(\pi x)\sin^2(\pi y)\), and \(\theta(\bx,t)=(\pi/4)\sin(\pi x)\cos(\pi y)+(\pi/8)\sin(2\pi t/T)\).  Thus the diffusion tensor is uniformly positive definite, but both the strength and the principal direction of uncertainty injection vary in space and time.

The computation uses the same parameters as in the isotropic double-gyre example: \(T=10\), \(\Delta t=0.02\), and a \(1024\times512\) grid.  The covariance equation is integrated along deterministic characteristics, with \(\bD\) evaluated at the current particle location and time.  The normalized moment-vFTLE is then computed from the largest eigenvalue of \(\bC_T\).  In this run, the normalized moment-vFTLE ranges from approximately \(1.01\times10^{-1}\) to \(9.54\times10^{-1}\).  The largest covariance eigenvalue reaches approximately \(1.92\times10^8\), while the smallest covariance eigenvalue remains positive, with \(\min_{\bx_i}\lambda_{\min}(\bC_T(\bx_i))\approx2.93\times10^{-2}\).  The largest anisotropy ratio is substantially higher than in the isotropic case, with \(\max\log_{10}(\lambda_{\max}/\lambda_{\min})\approx9.78\).  Hence the computed covariance tensor remains positive definite, while also becoming highly elongated in parts of the domain.

Figure~\ref{fig:double_gyre_anisotropic} shows the resulting scalar diagnostic and the dominant covariance eigenvectors.  The top panel displays the normalized moment-vFTLE field.  As in the isotropic case, the logarithmic scaling compresses a large range of covariance magnitudes and reveals ridge-like regions of strong uncertainty-induced spreading.  The bottom panel overlays a subsampled plot of the dominant eigenvector of \(\bC_T\) on the same scalar field.  These vectors should be interpreted as directions rather than oriented arrows, since the sign of an eigenvector is arbitrary.

Figure~\ref{fig:double_gyre_anisotropic_extra} further emphasizes the tensor information carried by \(\bC_T\).  The anisotropy field shows where the stochastic arrival covariance is close to circular and where it is strongly elongated.  We also compare the dominant covariance direction with the local dominant diffusion direction evaluated at the deterministic arrival point and final time.  The resulting angle difference is interpreted modulo \(\pi\), since both eigenvectors are unoriented; it ranges from nearly zero to almost \(\pi/2\), with a broad distribution over the domain.  This confirms that the final covariance direction is not determined solely by the local diffusion tensor.  Instead, it reflects the cumulative effect of anisotropic noise injection together with flow-induced stretching and rotation along the trajectory.

This example demonstrates why it is useful to compute the full covariance tensor rather than only a scalar variance indicator.  The moment-vFTLE identifies where stochastic spreading is large, while the eigenvalues and eigenvectors indicate how the stochastic arrival cloud is shaped and oriented.  In an anisotropic diffusion problem, this orientation is determined by the combined effect of anisotropic noise injection and flow-induced stretching and rotation.  Thus the covariance formulation retains directional information that is naturally lost when the stochastic spreading is reduced immediately to a scalar field.

\section{Conclusion}
\label{sec:conclusion}

We have developed a moment-based Eulerian approximation to variance-based finite-time Lyapunov exponents for stochastic flows in the small-noise regime.  The method is motivated by the observation that vFTLE depends on the stochastic arrival distribution only through its covariance.  Rather than evolving a full Fokker--Planck density for each initial point, we derive a closed covariance equation for the leading stochastic displacement and rewrite it as an Eulerian transport--reaction equation for a symmetric tensor field.  The covariance associated with an initial point is then obtained by evaluating this tensor field at the corresponding deterministic arrival location.

The formulation reduces the main computation to the evolution of \(d(d+1)/2\) scalar covariance components in \(d\) dimensions.  It is therefore substantially cheaper than a full PDF-based vFTLE computation when the diagnostic is required over many initial conditions.  The method also retains the full covariance tensor, so that both the scalar magnitude of stochastic spreading and the dominant spreading directions are available.  This directional information is particularly useful when the diffusion tensor is anisotropic or spatially varying.

The theoretical analysis shows that the covariance produced by the moment equation agrees, to leading order in the noise amplitude, with the covariance extracted from the Fokker--Planck density.  The same covariance framework also clarifies the connection with deterministic deformation: deterministic FTLE is recovered from isotropic initial covariance in the absence of process noise, whereas continuous process noise produces a time-integrated deformation covariance.  The numerical examples support the formulation: the linear test confirms the covariance component equations against an exact solution, the isotropic double-gyre benchmark produces structures strongly related to deterministic finite-time stretching, and the anisotropic double-gyre example demonstrates the additional directional information carried by the tensor formulation.

The method should be interpreted as a leading-order covariance approximation, rather than as a replacement for the full Fokker--Planck description in all regimes.  When the arrival distribution is strongly non-Gaussian, multimodal, or when higher-order uncertainty information is required, a full density-based method remains more appropriate.  Future work includes a high-order fully Eulerian semi-Lagrangian implementation of the covariance transport equation, more systematic comparison with Monte Carlo and Fokker--Planck vFTLE computations, and extensions to three-dimensional flows where the reduction from density evolution to covariance transport is expected to be especially advantageous.

\section*{Computational Methodology and Disclosure}

This research was developed with the assistance of GPT-5.5 (Thinking), which was used to support computational exploration, symbolic manipulation, mathematical derivation, literature organization, and technical drafting. The author acted as the principal investigator throughout the project, defining the research problems, setting the theoretical scope, directing the line of inquiry, evaluating intermediate outputs, and revising the final manuscript. The use of AI in this work should be understood as part of a human-led research process. AI-generated suggestions, calculations, and derivations were treated as provisional outputs subject to human judgment, correction, and refinement. The author retains full responsibility for the mathematical content, interpretation, conclusions, and presentation of the paper. Given the exploratory nature of this AI-assisted workflow, the author welcomes independent scrutiny of the arguments, derivations, references, and conclusions presented here. This article is disseminated as a preprint on arXiv in order to invite open discussion, feedback, and further verification by the mathematical community. It has not been submitted to a peer-reviewed journal in its current form.

\bibliographystyle{plain}
\bibliography{syleung}

@string{jcp = {J. Comput. Phys.}}

@article{Higham01,
	author = {Higham, D.J.},
	journal = {SIAM J. Numer. Anal.},
	pages = {525-546},
	title = {An algorithmic introduction to numerical simulation of stochastic differential equations},
	volume = 43,
	year = 2001}

@article{Kloeden92,
	author = {Kloeden, P.E. and Platen, E.},
	journal = {Springer},
	title = {Numerical solution of stochastic differential equations},
	year = 1992}

@article{ngyouleu22,
	author = {Ng, Y.K. and You, G. and Leung, S.},
	journal = {J. Comput. and Appl. Math.},
	number = {115255},
	title = {{Sparse subsampling of flow measurements for finite-time Lyapunov exponent in domains with obstacles}},
	volume = {431},
	year = {2023}}

@article{hal15,
	author = {Haller, G.},
	journal = {Annu. Rev. Fluid Mech.},
	pages = {137-162},
	title = {{L}agrangian coherent structures},
	volume = {47},
	year = {2015}}

@article{sfrs12,
	author = {Schneider, D. and Fuhrmann, J. and Reich, W. and Scheuermann, G.},
	journal = {Topological Methods in Data Analysis and Visualization II},
	publisher = {Springer-Verlag, New York},
	title = {A Variance Based {FTLE}-Like Method for Unsteady Uncertain Vector Fields},
	year = {2012}}

@article{lywn19,
	author = {Leung, S. and You, G. and Wong, T. and Ng, Y.K.},
	journal = {Proceedings of the Seventh International Congress of Chinese Mathematicians},
	number = {2},
	pages = {579-622},
	title = {Recent Developments in {E}ulerian Approaches for Visualizing Continuous Dynamical Systems},
	year = {2019}}

@article{youleu18,
	author = {You, G. and Leung, S.},
	journal = {J. Sci. Comput.},
	number = {1},
	pages = {70-96},
	title = {Eulerian Based Interpolation Schemes for Flow Map Construction and Line Integral Computation with Applications to Coherent Structures Extraction},
	volume = {74},
	year = {2018}}

@article{youleu20,
	author = {You, G. and Leung, S.},
	journal = {J. Sci. Comput.},
	number = {32},
	title = {{Fast Construction of Forward Flow Maps using Eulerian Based Interpolation Schemes}},
	volume = {82},
	year = {2020}}

@article{youleu21,
	author = {You, G. and Leung, S.},
	journal = {J. Comput. Phys.},
	number = {109905},
	title = {Computing the finite time Lyapunov exponent for flows with uncertainties},
	volume = {425},
	year = {2021}}

@article{leu13,
	author = {Leung, S.},
	journal = {Chaos},
	number = {043132},
	title = {The Backward Phase Flow Method for the Finite Time {L}yapunov Exponent},
	volume = {23},
	year = 2013}

@article{shalekmar05,
	author = {Shadden, S.C. and Lekien, F. and Marsden, J.E.},
	journal = {Physica D},
	pages = {271-304},
	title = {Definition and Properties of {L}agrangian Coherent Structures from Finite-Time {L}yapunov Exponents in Two-Dimensional Aperiodic Flows},
	volume = 212,
	year = 2005}

@article{halyua00,
	author = {Haller, G. and Yuan, G.},
	journal = {Physica D},
	pages = {352-370},
	title = {{L}agrangian Coherent Structures and Mixing in Two-Dimensional Turbulence},
	volume = 147,
	year = 2000}

@article{hal01,
	author = {Haller, G.},
	journal = {Physica D},
	pages = {248-277},
	title = {Distinguished Material Surfaces and Coherent Structures in Three-Dimensional Fluid Flows},
	volume = 149,
	year = 2001}

@article{hal02,
	author = {Haller, G.},
	journal = {Physics of Fluid},
	pages = {1851-1861},
	title = {{L}agrangian Coherent Structures from Approximate Velocity Data},
	volume = 14,
	year = 2002}

@article{leu11,
	author = {Leung, S.},
	journal = jcp,
	pages = {3500-3524},
	title = {An {E}ulerian Approach for Computing the Finite Time {L}yapunov Exponent},
	volume = 230,
	year = 2011}

\end{document}